\documentclass{article}

\usepackage{hyperref}
\hypersetup{%
colorlinks=true,
linkcolor=blue
}

\usepackage{amsmath,amsthm,amssymb}
\usepackage{mathrsfs,bm}
\usepackage{mathtools,tablists,enumitem}

\usepackage{expl3}
\ExplSyntaxOn
\cs_new_eq:NN \Repeat \prg_replicate:nn
\ExplSyntaxOff

\newcommand{\longhookrightarrow}{\mathrel{\lhook\joinrel\relbar\joinrel\rightarrow}}

\newcommand{\Op}{\mathrm{Op}}
\newcommand{\jbk}[1]{\left\langle {#1} \right\rangle}

\newcommand{\pprime}{{\prime\prime}}

\newcommand{\rmop}[1]{\mathop{\mathrm{#1}}}

\newcommand{\e}{\mathrm{e}}
\newcommand{\iu}{\mathrm{i}}

\newcommand{\df}{\mathrm{d}}

\newcommand{\nv}{\nu}

\renewcommand{\Re}{\rmop{Re}}

\newcommand{\p}{\partial}

\newcommand{\SL}[1][\partial\Omega]{\mathcal{S}_{#1}}
\newcommand{\NP}[1][\partial\Omega]{\mathcal{K}_{#1}}

\newcommand{\Cbb}{\mathbb{C}}

\newcommand{\Hbb}{\mathbb{H}}

\newcommand{\Nbb}{\mathbb{N}}

\newcommand{\Rbb}{\mathbb{R}}

\newcommand{\Zbb}{\mathbb{Z}}

\newcommand{\Dscr}{\mathscr{D}}

\newcommand{\Sscr}{\mathscr{S}}

\newcommand{\Acal}{\mathcal{A}}
\newcommand{\Bcal}{\mathcal{B}}

\newcommand{\Fcal}{\mathcal{F}}

\newcommand{\Hcal}{\mathcal{H}}

\newcommand{\Kcal}{\mathcal{K}}
\newcommand{\Lcal}{\mathcal{L}}

\newcommand{\Rcal}{\mathcal{R}}
\newcommand{\Scal}{\mathcal{S}}
\newcommand{\Tcal}{\mathcal{T}}
\newcommand{\Ucal}{\mathcal{U}}

\newcommand{\GS}{\Sigma}

\numberwithin{equation}{section}

\newtheorem{prop}{Proposition}[section]
\newtheorem{theo}[prop]{Theorem}
\newtheorem{coro}[prop]{Corollary}
\newtheorem{lemm}[prop]{Lemma}

\theoremstyle{definition}

\newtheorem*{note*}{Note}
\newtheorem*{claim*}{Claim}
\newtheorem*{exam*}{Example}
\newtheorem*{rema*}{Remark}
\newtheorem*{exer*}{Exercise}

\renewcommand{\nv}{n}

\renewcommand{\iu}{i}
\renewcommand{\e}{e}
\renewcommand{\df}{d}

\newcommand{\Kell}{\mathbf{K}}
\newcommand{\Ell}{\mathbf{E}}

\mathtoolsset{showonlyrefs=true}

\usepackage{tikz}
\usetikzlibrary{math, calc, arrows.meta}

\author{Shota Fukushima\thanks{Department of Mathematics and Institute of Applied Mathematics, Inha University, Incheon 22212, S. Korea. \\
E-mail: \nolinkurl{shota.fukushima.math@gmail.com} (S. Fukushima), \nolinkurl{hbkang@inha.ac.kr} (H. Kang)} \and Hyeonbae Kang\footnotemark[2]}
\title{Spectral structure of the Neumann-Poincar\'e operator on axially symmetric functions\thanks{This work was supported by NRF of S. Korea grant No. 2022R1A2B5B01001445.}}

\begin{document}

\maketitle

\begin{abstract}
    We consider the Neumann-Poincar\'e operator on a three-dimensional axially symmetric domain which is generated by rotating a planar domain around an axis which does not intersect the planar domain. We investigate its spectral structure when it is restricted to axially symmetric functions. If the boundary of the domain is smooth, we show that there are infinitely many axially symmetric eigenfunctions and derive Weyl-type asymptotics of the corresponding eigenvalues. We also derive the leading order terms of the asymptotic limits of positive and negative eigenvalues. The coefficients of the leading order terms are related to the convexity and concavity of the domain. If the boundary of the domain is less regular, we derive decay estimates of the eigenvalues. The decay rate depends on the regularity of the boundary. We also consider the domains with corners and prove that the essential spectrum of the Neumann-Poincar\'e operator on the axially symmetric three-dimensional domain is non-trivial and contains that of the planar domain.
\end{abstract}



\tableofcontents

\section{Introduction}

This is a study of the spectral structure of the Neumann-Poincar\'e operator on axially symmetric domains when the operator is restricted on the axially symmetric functions. This study is motivated by the search of three-dimensional domains where the cloaking by anomalous localized resonance takes place.

Let $\Omega\subset \Rbb^d$ ($d=2, 3$) be a bounded Lipschitz domain. The Neumann-Poincar\'e operator $\NP^*$ (abbreviated to NP operator) is the integral operator on $\partial \Omega$ defined by
\begin{equation}\label{eq_nps_defi}
    \NP^*[f](X):=\frac{1}{\omega_d}\rmop{p.v.}\int_{\partial\Omega} \frac{(X-Y)\cdot \nv_X}{|X-Y|^d}f(Y)\, \df \sigma (Y) \quad (X\in \partial\Omega),
\end{equation}
where $\omega_d$ is the area of the unit sphere in $\Rbb^d$, namely, $2\pi$ if $d=2$ and $4\pi$ if $d=3$.
This operator is realized as a self-adjoint operator on $H^{-1/2}(\partial \Omega)$ ($L^2$-Sobolev space of order $-1/2$) by introducing an inner product defined in terms of the single layer potential (see \eqref{eq_inner_product}). Thus its spectrum consists of the essential spectrum and eigenvalues of finite multiplicities.

Many significant results on the spectral properties of the NP operator have been obtained in recent years. When the boundary $\partial\Omega$ is $C^{1,\alpha}$ for some $\alpha>0$, then $\NP^*$ is compact and the spectrum consists of eigenvalues accumulating at $0$. If $d=3$ and $\partial \Omega$ is $C^{2,\alpha}$, the Weyl-type asymptotics of the eigenvalue decay is obtained in \cite{Miyanishi22}. In particular, the decay rate is $j^{-1/2}$ when eigenvalues $\{\lambda_j=\lambda_j (\NP^*)\}$ are enumerated in such a way that $|\lambda_1| \ge |\lambda_2| \ge \ldots$. If $\partial \Omega$ is $C^\infty$, then the Weyl-type asymptotics of negative and positive eigenvalues are obtained in \cite{Miyanishi-Rozenblum19}. If $d=3$ and $\partial \Omega$ is $C^{1,\alpha}$, the decay rate of $j^{-\alpha/2+\varepsilon}$ for any $\varepsilon>0$ is proved in \cite{FKMdr}. It is shown in \cite{HP1, HP2} that some three-dimensional domains with Lipschitz boundaries have nontrivial essential spectrum. The spectral structure of the NP operator for the two-dimensional case is quite different from that for the three-dimensional case. It is proved in \cite{FKMdr} that if $\Omega$ is a bounded domain in $\Rbb^{2}$ with $C^{k, \alpha}$ boundary, then
    \begin{equation}\label{1-700}
        |\lambda_j (\NP^*)|=o(j^{-k+1 -\alpha+\varepsilon}) \quad (j\to \infty)
    \end{equation}
for any $\varepsilon>0$. We emphasize for comparison with the three-dimensional case that the decay gets faster as $k$ gets larger. It is proved in \cite{Perfekt-Putinar17} that if $\Omega$ is a curvilinear polygon with the interior angle(s) $\alpha_1, \ldots, \alpha_N\in (0, 2\pi)$ and 
\begin{equation}\label{eq_b}
    b:=\max \left\{ \left| \frac{1}{2}-\frac{\alpha_1}{2\pi}\right|, \ldots, \left| \frac{1}{2}-\frac{\alpha_N}{2\pi}\right|\right\}, 
\end{equation}
then
    \begin{equation}\label{1-800}
        \sigma_\mathrm{ess}(\NP^*) =[-b, b].
    \end{equation}

In this paper we deal with the NP operator on the axially symmetric domains which are generated by rotating a planar domain around an axis which does not intersect the domain. On such domains, the NP operator is decomposed in terms of the Fourier modes. We investigate the spectral properties of the zeroth mode operator which is the NP operator restricted to axially symmetric functions.

The investigation of this paper is strongly motivated by the study of the cloaking by anomalous localized resonance (CALR). To make the presentation short, we do not include in this paper the mathematical formulation and the physical meaning of CALR and simply refer to \cite{ACKLM13, MM} for them. It is known that CALR occurs on annuli \cite{ACKLM13, KLSW, MN_06} and ellipses \cite{AK} in two dimensions, and it does not occur on balls \cite{ACKLM13} and on strictly convex domains in three dimensions \cite{AKMN21}. Not a single three-dimensional domain where CALR occurs has been discovered. The occurrence (and nonoccurrence) of CALR is determined by two spectral properties of the NP operator: decay rate of eigenvalues and localization of surface plasmons which are single layer potentials of eigenfunctions. In three dimensions the decay rate of NP eigenvalues is $j^{-1/2}$ if the boundary of the domain is $C^{2,\alpha}$ as mentioned before. It is proved in \cite{AKMN21} that if the domain is strictly convex, then the single layer potentials of eigenfunctions decays too fast outside the surface resulting in nonoccurrence of CALR. In the same paper, it is observed by numerical computations that something different occurs on the torus. The single layer potentials of axially symmetric eigenfunctions are relatively large and  they show a clear wave pattern outside the domain different from that of non-axially symmetric eigenfunctions. Motivated by these observations related with CALR on three dimensions, we investigate, as the starting point of the investigation on the possibility of CALR on three-dimensional domains, the decay rate or asymptotic behavior of eigenvalues of the NP operator restricted to axially symmetric functions on tori or more generally on axially symmetric domains.

After rotation and translation if necessary, we assume with no loss of generality that the axis of the axially symmetric domain $\Omega\subset \Rbb^3$ is $x$-axis. Then $\Omega$ is of the form
\begin{equation}\label{Omega_def}
  \Omega=\{ (x, y\cos \varphi, y\sin \varphi)\in \Rbb^3 \mid (x, y)\in \Sigma, \ \varphi \in S^1 \}
\end{equation}
where $\Sigma\subset \Rbb^2$ is a bounded domain with the Lipschitz boundary such that $\overline{\Sigma}\subset \Rbb \times (0, \infty)$ and $S^1=\Rbb/2\pi \Zbb$ is the circle. For example, if $\Sigma$ is a disc, then $\Omega$ is a solid torus.
Figure \ref{fg_omega_torus} shows a typical example of the axially symmetric domain.

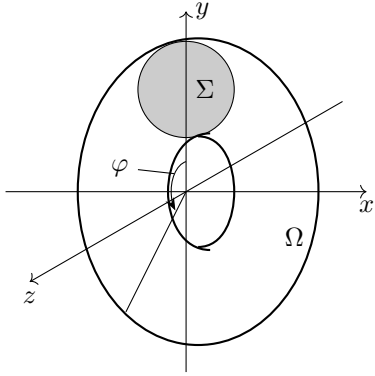
\begin{figure}[htb]
    \centering
    \tikzmath{\rad=0.8;}
\begin{tikzpicture}[scale=0.8]
    \coordinate (centerSigma) at (0, 1.7);
    \coordinate (centerTorus) at (0.2, 0);
    \node (x) at (3, 0) [below] {$x$};
    \node (y) at (0, 3) [right] {$y$};
    \node (z) at (-2.6, -1.5) [below] {$z$};
    \node (Omega) at (1.8, -0.75) {$\Omega$};
    \fill[black!20] (centerSigma) circle [radius=\rad];
    \draw (centerSigma) circle [radius=\rad];
    \node (Sigma) at (centerSigma) [right] {$\Sigma$};
    \draw[thick] (centerTorus) circle [x radius=2, y radius=2.55];
    \draw[thick] (0.4, 0.97) to [out=180, in=90] (-0.3, 0);
    \draw[thick] (0.4, -0.97) to [out=180, in=270] (-0.3, 0);
    \draw[thick] (0.2, 0.92) to [out=0, in=90] (0.8, 0);
    \draw[thick] (0.2, -0.92) to [out=0, in=270] (0.8, 0);
    \draw (0, 0)--(-1, -2);
    \draw[->, >=Stealth] (0, 0.5) arc [start angle=100, end angle=220, x radius=0.3, y radius=0.5];
    \node (eta) at (-0.8, 0.4) [left] {$\varphi$};
    \draw (eta)--(-0.2, 0.3);
    \draw[->] (-3, 0)--(3, 0);
    \draw[->] (0, -3)--(0, 3);
    \draw[->] (2.6, 1.5)--(-2.6, -1.5);
\end{tikzpicture}
\caption{The axially symmetric domain $\Omega$ generated by rotating $\Sigma$}
\label{fg_omega_torus}
\end{figure}

We say that a function $f$ on $\partial\Omega$ is axially symmetric if $f(x, y\cos \varphi, y\sin \varphi)$
is independent of $\varphi\in S^1$. We investigate the spectral properties of the NP operator $\NP^*$ on such functions.
We deal with not only domains with smooth boundaries but also Lipschitz boundaries in relation to the essential spectrum.

The first result of this paper is that there are infinitely many NP eigenvalues with axially symmetric eigenfunctions if the boundary of the domain is smooth enough as the following theorem states:

\begin{theo}\label{theo_self_adjoint_zeroth}
    If $\partial\Sigma$ is $C^{1, \alpha}$ for some $\alpha>0$, then $\NP^*$ has infinitely many real eigenvalues with axially symmetric eigenfunctions in $H^{-1/2}(\partial\Omega)$.
\end{theo}

We enumerate all eigenvalues $\{\rho_j(\NP^*)\}_{j=1}^\infty$ with axially symmetric eigenfunctions in the descending order
\[
    |\rho_1(\NP^*)|\geq |\rho_2(\NP^*)|\geq |\rho_3(\NP^*)|\geq \cdots \to 0.
\]
Then, we can estimate the decay rate of such eigenvalues from above. The decay rates are different depending on regularity of the boundary.
\begin{theo}
    \label{theo_eigenvalue_C1a_C2a}
    \begin{enumerate}[label=(\roman*)]
        \item \label{enum_eigenvalue_C1a}If $\partial\Sigma$ is $C^{1, \alpha}$ for some $\alpha>0$, then it holds that
        \begin{equation}\label{1-100}
            |\rho_j (\NP^*)|=o(j^{-\alpha+\varepsilon}) \quad (j\to \infty)
        \end{equation}
        for any $\varepsilon>0$.
        \item \label{enum_eigenvalue_C2a}If $\partial\Sigma$ is $C^{2, \alpha}$ for some $\alpha>0$, then it holds that
        \begin{equation}\label{1-200}
            |\rho_j (\NP^*)|=O(j^{-1}) \quad (j\to \infty).
        \end{equation}
    \end{enumerate}
\end{theo}

The decay rate in \eqref{1-100} is like that of the NP eigenvalues on two-dimensional domains with $C^{1, \alpha}$ boundaries as described in \eqref{1-700}. This is so because the NP operator restricted to axially symmetric functions behaves like the NP operator on the section $\partial\Sigma$ on $C^{1, \alpha}$ domains. In fact, the NP operator restricted to axially symmetric functions is decomposed into three pieces: one is the NP operator on the section $\partial\Sigma$, another one is essentially the single layer potential on $\partial\Sigma$, and the last one is a remainder term which is smoothing (see \eqref{eq_operator_decomposition}). On $C^{1,\alpha}$-domain, the integral kernel of the NP operator on $\partial\Sigma$ has stronger singularity, namely, $|p-p^\prime|^{-1+\alpha}$, than that of the single layer potential on $\partial\Sigma$, namely, $\log |p-p^\prime|$. Thus, the NP operator on $\partial\Sigma$ is responsible for the spectral properties of the NP operator on $\partial\Omega$ yielding the decay rate \eqref{1-100}. If $\partial\Omega$ is $C^{2, \alpha}$, this relation becomes opposite and the principal term for the spectral properties is the single layer potential term. This difference causes the difference of the decay rates in \eqref{1-100} and \eqref{1-200}.

If $\partial\Omega$ is $C^\infty$, then we can improve the decay estimate in \eqref{1-200}. Moreover, we can find the leading order terms of the asymptotic behavior of the positive and negative eigenvalues. We enumerate all positive eigenvalues $\rho^+_j(\NP^*)>0$ and negative eigenvalues $-\rho^-_j(\NP^*)<0$ with axially symmetric eigenfunctions as
\[
    \rho^\pm_1(\NP^*)\geq \rho^\pm_2(\NP^*)\geq \rho^\pm_3(\NP^*)\geq \cdots.
\]
The following theorem exhibits the asymptotic behaviors of these eigenvalues. In what follows, we denote the outward unit normal vector at $p\in \partial\Sigma$ by ${\widetilde\nv}_p$. We then define
\begin{equation}\label{def_vp}
v_p:= -{\widetilde\nv}_{p,2}
\end{equation}
where ${\widetilde\nv}_{p,2}$ is the second component of ${\widetilde\nv}_p$. It is worth mentioning that $v_p$ is the first component of the unit tangential vector with the positive orientation at $p\in \partial\Sigma$.
\begin{theo}
    \label{theo_np_ev_rot}
    If $\partial\Sigma$ is $C^\infty$, then $\NP^*$ has infinitely many positive and negative eigenvalues with axially symmetric eigenfunctions. Moreover, the absolute values of eigenvalues and positive and negative eigenvalues have the asymptotic behaviors
    \begin{equation}
        \label{eq_asymptotic_rot}
        |\rho_j(\NP^*)|\sim C_0 j^{-1}, \quad
        \rho^\pm_j (\NP^*)\sim C^\pm_0 j^{-1}
    \end{equation}
    where $C_0$ and $C^\pm_0$ are defined by
    \begin{align}
        C^\pm_0 = \mp \frac{1}{4\pi}\int_{\{ p=(x,y) \in \partial\Sigma \mid \mp v_p >0 \}} \frac{v_p}{y}\, \df \sigma (p) \label{eq_Cpm0}
    \end{align}
    and
    \begin{equation}\label{eq_C0}
    C_0=C^+_0+C^-_0 .
    \end{equation}
\end{theo}

To clarify the geometric meaning of $C^\pm_0$, let us assume that $\GS$ is convex. Because of \eqref{def_vp}, if $v_p>0$, then ${\widetilde\nv}_p$ is downward, and hence $(x, y\cos \varphi, y\sin \varphi)$ is a concave point on $\partial\Omega$. So, $C^-_0$ is the integration over the concave part of $\GS$; $C^+_0$ over the convex part. The connection between negative NP eigenvalues and concavity of the domain has been discovered in \cite{JK, Miyanishi-Rozenblum19}.

Since $C^\pm_0>0$, we infer that there are infinitely many both positive and negative NP eigenvalues with axially symmetric eigenfunctions. Existence of infinitely many negative eigenvalues on tori is proved in \cite{AJKKM19} by considering non-axially symmetric functions and employing the stationary phase method for the high oscillation limit with respect to the $\varphi$-direction. Theorem \ref{theo_np_ev_rot} shows that the NP operator on the axially symmetric functions, which do not oscillate in the $\varphi$-direction, have infinitely many negative eigenvalues. More general sufficient condition for the existence of infinitely many negative NP eigenvalues is given in \cite{Miyanishi-Rozenblum19}: if there has a point on the boundary where at least one principal curvature is positive, then the NP operator has infinitely many negative eigenvalues. In fact, our strategy to the proof of Theorem \ref{theo_np_ev_rot} follows that in \cite{Miyanishi-Rozenblum19}, namely, calculus of pseudodifferential operators and invoking the result of \cite{Birman-Solomjak77}.

It is interesting to observe that the difference $C_0^- - C_0^+$ is the volume of $\Sigma$ in the hyperbolic space $\Hbb^2:=(\Rbb\times (0, \infty), (\df x^2+\df y^2)/y^2)$. In fact, it follows from \eqref{eq_Cpm0} and the Stokes theorem that
\[
C_0^- - C_0^+ = \frac{1}{4\pi}\int_{\partial\Sigma} \frac{\df x}{y} = \frac{1}{4\pi}\int_\Sigma \frac{\df x \df y}{y^2}.
\]
In particular, we have $C_0^+ < C_0^-$.

We also consider the case when $\partial\Sigma$ is Lipschitz. In this case, the NP operator $\NP[\partial\Sigma]^*$ on $\partial\Sigma$ may have a nontrivial essential spectrum. We prove an inclusion relation of essential spectra of NP operators on $\partial\Omega$ and $\partial\Sigma$.

\begin{theo}
    \label{theo_essential_spectrum}
    If $\partial\Sigma$ is Lipschitz, then the inclusion relation
    \begin{equation}
        \label{eq_essential_spectrum_inclusion}
        \sigma_\mathrm{ess}(\NP[\partial\Sigma]^*, H^{-1/2}(\partial\Sigma))
        \subset \sigma_\mathrm{ess}(\NP^*, H^{-1/2}(\partial\Omega))
    \end{equation}
    holds.
\end{theo}

Theorem \ref{theo_essential_spectrum} gives examples of three-dimensional domains with non-trivial essential spectra. For instance, we apply Theorem \ref{theo_essential_spectrum} together with \eqref{1-800} of \cite{Perfekt-Putinar17} to obtain the following corollary:

\begin{coro}
    \label{coro_curvilinear_polygon}
    If $\partial\Sigma$ is a $C^2$-smooth curvilinear polygon with the interior angle(s) $\alpha_1, \ldots, \alpha_N \in (0, 2\pi)$ and define $b>0$ by \eqref{eq_b}, then
    \begin{equation}
        \label{eq_essential_spectrum_curvilinear_polygon}
        [-b, b]
        \subset \sigma_\mathrm{ess}(\NP^*, H^{-1/2}(\partial\Omega)).
    \end{equation}
\end{coro}

A few remarks on the inner product and the symmetrization of the NP operator are in order. The Sobolev space $H^{-1/2}(\partial\Omega)$ is equipped with the inner product
\begin{equation}
    \label{eq_inner_product}
    \jbk{f, g}_*:=-\int_{\partial\Omega} f(X){\SL} [g](X)\, \df \sigma (X),
\end{equation}
where ${\SL}$ is the single layer potential, namely,
\begin{equation}
    \label{eq_single_layer}
    {\SL} [f](X):= \int_{\partial\Omega} \Gamma(X-Y) f(Y)\, \df \sigma (Y),
\end{equation}
where $\Gamma(X)$ is the fundamental solution to the Laplacian: it is $(2\pi)^{-1}\log |X|$ if $d=2$ and $-(4\pi |X|)^{-1}$ if $d=3$. In three dimensions $\SL$ is invertible as acting from $H^{-1/2}(\partial\Omega)$ onto $H^{1/2}(\partial\Omega)$, and hence $\jbk{\cdot , \cdot }_*$ is actually an inner product. However, in two dimensions, there is a domain $\Omega$ such that $\SL$ has a non-trivial kernel. But, in such a case, if we dilate the domain, the single layer potential on dilated domains are invertible. Since the NP operator is invariant under the dilation of domains, we assume without loss of generality that $\jbk{\cdot , \cdot }_*$ is actually an inner product on $H^{-1/2}(\partial\Omega)$.

The NP operator $\NP^*$ on a bounded Lipschitz domain is realized as a self-adjoint operator on $H^{-1/2}(\partial\Omega)$ equipped with the inner product $\jbk{\ , \ }_*$. It is a consequence of the Plemelj's symmetrization principle
\begin{equation}\label{eq_Plemelj}
    {\SL}\NP^*=\NP {\SL},
\end{equation}
where $\NP$ is $L^2$-adjoint of $\NP^*$, namely,
\begin{equation}\label{eq_np_defi}
    {\NP} [f](X):=-\frac{1}{\omega_d}\rmop{p.v.}\int_{\partial\Omega} \frac{(X-Y)\cdot \nv_Y}{|X-Y|^d}f(Y)\, \df \sigma (Y) \quad (X\in \partial\Omega),
\end{equation}
(see \cite{KPS07} for example). This operator is also called NP operator.

Finally, we emphasize that the subject of this paper has no intersection with that of \cite{Ji-Kang23MA} whose title is somewhat similar to this one; there the spectral structure of $m$-fold symmetric domains is studied. A two-dimensional domain is $m$-fold symmetric where $m\geq 2$ is an integer if it is invariant under the rotation by the angle $2\pi/m$.

This paper is organized as follows. In Section \ref{sec:decom}, we decompose the NP operator into Fourier modes. Section \ref{sec:zeroth} is for analysis of the zeroth mode operator which is the NP operator on axially symmetric functions.
Sections of the rest are to prove theorems presented in Introduction. Appendix \ref{sec_proof_elliptic_behavior} is to prove asymptotic formulas for certain elliptic integrals which appear in the course of proofs.

Throughout this paper, notation $A\lesssim B$ means that there exists a constant $C>0$ such that $A\leq CB$. The meaning of $A\gtrsim B$ is analogous, and $A \approx B$ means both $A\lesssim B$ and $A\gtrsim B$ hold.

\section{Decomposition into modes}\label{sec:decom}

\subsection{Decomposition of the NP operator}
Let $\Omega$ be the axially symmetric domain defined by \eqref{Omega_def}. Consider the diffeomorphism
\[
    \Psi: \partial \Sigma \times S^1 \longrightarrow \partial\Omega, \quad (p, \varphi) \longmapsto X=(x, y\cos \varphi, y\sin \varphi) \ \ (p=(x, y)),
\]
and the linear operator
\begin{equation}\label{eq_defi_j}
    \Ucal[f](p, \varphi):=y^{1/2}f(\Psi(p, \varphi)) \quad (p=(x, y)\in \partial\Sigma, \, \varphi\in S^1)
\end{equation}
for $f: \partial\Omega\to \Cbb$. Since the surface element $\df \sigma(X)$ on $\p\Omega$ is given by
\begin{equation}\label{dsigma}
\df \sigma(X) = y\,\df \sigma (p) \df \varphi \quad (X=\Psi (p, \varphi)),
\end{equation}
where $\df \sigma$ in the right hand side is the line element on $\p\Sigma$, the operator $\Ucal: L^2(\partial\Omega)\to L^2(\partial\Sigma \times S^1)$ is a unitary operator, that is, an isomorphism with the isometric property
\begin{equation}\label{eq_J_unitary}
    \int_{\partial\Omega} f(X) \overline{g(X)} \, \df \sigma(X)
    =\int_{S^1} \int_{\partial\Sigma} \Ucal [f](p, \varphi) \overline{\Ucal [g](p, \varphi)} \, \df \sigma (p) \df \varphi
\end{equation}
for all $f, g \in L^2(\partial\Omega)$. Moreover, $\Ucal$ is a bounded invertible operator from $H^s (\partial\Omega)$ onto $H^s (\partial\Sigma \times S^1)$ for all $s\in [-1, 1]$. This can be seen by direct computations for $s=1$, by duality for $s=-1$, and by interpolation for $s \in (-1,1)$. Here and throughout the paper the circle $S^1$ is identified with $[-\pi, \pi)$.

We introduce functions
\begin{equation}\label{eq_AB}
    \begin{aligned}
        A_k(\delta):&=\delta^2 \int_0^{\pi/2} \frac{\cos (2k\varphi)}{(\delta^2+\sin^2 \varphi)^{3/2}}\, \df \varphi, \\
        B_k(\delta):&=\int_0^{\pi/2}\frac{\cos (2k\varphi)\sin^2 \varphi}{(\delta^2+\sin^2 \varphi)^{3/2}}\, \df \varphi,
    \end{aligned}
    \quad (\delta>0)
\end{equation}
for $k=0,1,2, \ldots$, and the distance-like quantity
\begin{equation}\label{eq_delta}
    \delta (p, p^\prime):=\frac{|p-p^\prime|}{2(yy^\prime)^{1/2}} \quad (p=(x, y), \ p^\prime =(x^\prime, y^\prime)\in \partial\Sigma).
\end{equation}

We will use the following lemma in the further part of this paper. 

\begin{lemm}\label{lemm_triangular_like}
    For any compact subset $K\subset \Rbb \times (0, \infty)$, there exists a constant $C>0$ such that the inequality
    \begin{equation}
        \label{eq_triangular_like}
        C^{-1}|\delta (p, r)-\delta (r, q)|\leq \delta (p, q)\leq C(\delta (p, r)+\delta (r, q))
    \end{equation}
    holds for all $p, q, r\in K$.
\end{lemm}

\begin{proof}
    It suffices to consider $K=[a, b]\times [c, d]$ ($c>0$). Let $p=(x, y)$, $q=(x^\prime, y^\prime)$, $r=(x^\pprime, y^\pprime)\in K$. Then we have
    \begin{align*}
        \delta (p, q)&=\frac{|p-q|}{2(yy^\prime)^{1/2}}
        \leq \frac{|p-r|}{2(yy^\prime)^{1/2}}+\frac{|r-q|}{2(yy^\prime)^{1/2}} \\
        &=\left(\frac{y^\pprime}{y^\prime}\right)^{1/2}\frac{|p-r|}{2(yy^\pprime)^{1/2}}+\left(\frac{y^\pprime}{y}\right)^{1/2}\frac{|r-q|}{2(y^\pprime y^\prime)^{1/2}} \\
        &\leq \left(\frac{d}{c}\right)^{1/2}(\delta (p, r)+\delta (r, q)),
    \end{align*}
which proves the second inequality in \eqref{eq_triangular_like}.

    On the other hand, we have
    \begin{align*}
        \delta (p, r)-\delta (r, q)
        &\leq \frac{|p-q|+|r-q|}{2(yy^\pprime)^{1/2}}-\frac{|r-q|}{2(y^\pprime y^\prime)^{1/2}} \\
        &\leq \left(\frac{d}{c}\right)^{1/2}\delta (p, q)+\frac{1}{2c^{1/2}}|r-q|\left|\frac{1}{y^{1/2}}-\frac{1}{(y^\prime)^{1/2}}\right| \\
        &=\left(\frac{d}{c}\right)^{1/2}\delta (p, q)+\frac{((b-a)^2+(d-c)^2)^{1/2}}{c^{1/2}(y^{1/2}+(y^\prime)^{1/2})}\frac{|y-y^\prime|}{2(yy^\prime)^{1/2}} \\
        &\leq \left(\left(\frac{d}{c}\right)^{1/2}+\frac{((b-a)^2+(d-c)^2)^{1/2}}{2c}\right)\delta (p, q).
    \end{align*}
    By changing roles of $p$ and $q$, we obtain the same estimate for $\delta (r, q)-\delta (p, r)$ and hence 
    \[
        |\delta (p, r)-\delta (r, q)|\leq \left(\left(\frac{d}{c}\right)^{1/2}+\frac{((b-a)^2+(d-c)^2)^{1/2}}{2c}\right)\delta (p, q).
    \]
    This proves the first inequality in \eqref{eq_triangular_like}.
\end{proof}

For a function $f(p)$ on $\partial\Sigma$ and a function $\eta(\varphi)$ on $S^1$, we define the tensor product $f\otimes \eta: \partial\Sigma \times S^1\to \Cbb$ by
\[
    (f\otimes \eta)(p, \varphi):=f(p)\eta(\varphi).
\]
Also, we define a function
\[
    e_k (\varphi):=\e^{\iu k \varphi}
\]
for $k\in \Zbb$ and $\varphi\in S^1$. Then for each $k \in \Zbb$ and $s \in [-1,1]$, there are constants $c_{k,s}, \ C_{k,s}$ such that
\begin{equation}\label{Cks}
c_{k,s} \| f \|_{H^s(\p\Sigma)} \le \| f \otimes e_k \|_{H^s(\p\Omega)} \le C_{k,s} \| f \|_{H^s(\p\Sigma)}
\end{equation}
for all $f \in H^s(\p\Sigma)$.

We denote the integral kernel of the two-dimensional NP operator on $\partial\Sigma$ by $K^*_{\partial\Sigma}(p, p^\prime)$, namely,
\begin{equation}
    \label{eq_nps_2D}
    K^*_{\partial\Sigma}(p, p^\prime):=\frac{1}{2\pi}\frac{(p-p^\prime)\cdot \widetilde \nv_p}{|p-p^\prime|^2}.
\end{equation}

\begin{prop}
    \label{prop_nps_fourier}
    For $f\in H^{-1/2}(\partial\Sigma)$ and $k\in \Zbb$, we have
    \begin{equation}
        \label{eq_nps_separated}
            \Ucal\NP^* \Ucal^{-1}[f\otimes e_k]
            =\Kcal_k^*[f]\otimes e_k
    \end{equation}
    where $\Kcal_k^*$ is the integral operator of the form
    \[
        \Kcal_k^*[f](p):=\int_{\partial\Sigma} K^*_k(p, p^\prime)f (p^\prime)\,\df \sigma(p^\prime)
    \]
    with
    \begin{equation}
        \label{eq_ak_representation}
        K^*_k(p, p^\prime)=K^*_{\partial\Sigma}(p, p^\prime)A_k(\delta (p, p^\prime))-\frac{v_p}{4\pi y}B_k(\delta (p, p^\prime)).
    \end{equation}
    Similarly, for the NP operator $\NP$ (see \eqref{eq_np_defi}), we have
    \begin{equation}
        \label{eq_np_separated}
        \Ucal\NP \Ucal^{-1}[f\otimes e_k](p, \varphi)
            =\Kcal_k[f]\otimes e_k
    \end{equation}
    where
    \begin{equation}\label{eq_kk_representation}
        \Kcal_k [f](p):=\int_{\partial\Sigma} K_k^*(p^\prime, p) f(p^\prime)\, \df \sigma (p^\prime).
    \end{equation}
\end{prop}

\begin{proof}
Let $g$ be a function on $\p\Sigma \times S^1$. For $X, X^\prime \in \partial\Omega$, let $X=\Psi(p,\varphi)$ and $X^\prime=\Psi(p^\prime,\varphi^\prime)$. Thanks to \eqref{dsigma}, we have
\begin{align*}
    &\Ucal\NP^* \Ucal^{-1} [g](p, \varphi) \\
    &=\frac{y^{1/2}}{4\pi}\int_{\partial\Omega}
    \frac{(\Psi(p,\varphi)-X')\cdot \nv_{\Psi(p,\varphi)}}{|\Psi(p,\varphi)-X^\prime|^3} \Ucal^{-1} [g](X^\prime) \df \sigma(X^\prime) \\
    &=\frac{1}{4\pi}\int_{\partial\Sigma \times S^1}
    \frac{(\Psi(p,\varphi)-\Psi(p^\prime,\varphi^\prime))\cdot \nv_{\Psi(p,\varphi)}}{|\Psi(p,\varphi)-\Psi(p^\prime,\varphi^\prime)|^3}(y y^\prime)^{1/2}
    g(p^\prime, \varphi^\prime)\, \df \sigma (p^\prime) \df \varphi^\prime.
\end{align*}
Straight-forward calculations yield
\begin{equation}\label{2-1000}
|\Psi(p,\varphi)-\Psi(p^\prime,\varphi^\prime)|^2 = |p-p^\prime|^2+4yy^\prime \sin^2 \frac{\varphi-\varphi^\prime}{2}.
\end{equation}
If we denote the outward normal vector at $p \in \partial\Sigma$ by $\widetilde{n}_p=(\nv_{p,1}, \nv_{p,2})$ as before, then
\[
\nv_{\Psi(p,\varphi)} = (\nv_{p,1}, \nv_{p,2} \cos \varphi, \nv_{p, 2} \sin \varphi),
\]
and hence
\[
(\Psi(p,\varphi)-\Psi(p^\prime,\varphi^\prime))\cdot \nv_{\Psi(p,\varphi)} = (p-p^\prime)\cdot {\widetilde\nv}_p+ 2\nv_{p,2} y^\prime \sin^2 \frac{\varphi-\varphi^\prime}{2}.
\]
Since $\nv_{p,2}=-v_p$, we have
\begin{align*}
    &\Ucal\NP^* \Ucal^{-1} [g](p, \varphi) \\
    &=\frac{1}{4\pi}\int_{\partial\Sigma \times S^1}
    \frac{(p-p^\prime)\cdot {\widetilde\nv}_p-2v_p y^\prime\sin^2 \frac{\varphi-\varphi^\prime}{2}}{(|p-p^\prime|^2+4yy^\prime \sin^2 \frac{\varphi-\varphi^\prime}{2})^{3/2}}(y y^\prime)^{1/2}
    g(p^\prime, \varphi^\prime)\, \df \sigma (p^\prime) \df \varphi^\prime.
\end{align*}

If we set $g=f\otimes e_k$, that is, $g(p, \varphi)=f(p)\e^{\iu k \varphi}$,  then we have
\begin{align*}
    &\Ucal\NP^* \Ucal^{-1} [f\otimes e_k](p, \varphi) \\
    &=\frac{1}{4\pi}\int_{\partial\Sigma}\df \sigma (p^\prime)\, f(p^\prime)(y y^\prime)^{1/2}
    \int_{S^1} \df \varphi^\prime\,
    \frac{(p-p^\prime)\cdot {\widetilde\nv}_p-2v_p y^\prime\sin^2 \frac{\varphi-\varphi^\prime}{2}}{(|p-p^\prime|^2+4yy^\prime \sin^2 \frac{\varphi-\varphi^\prime}{2})^{3/2}}
    \e^{\iu k \varphi^\prime}   \\
    &=\frac{1}{4\pi} \e^{\iu k \varphi}\int_{\partial\Sigma}\df \sigma (p^\prime)\,  f (p^\prime) (y y^\prime)^{1/2}\int_{S^1} \df \varphi^\prime\,
    \frac{(p-p^\prime)\cdot {\widetilde\nv}_p-2v_p y^\prime\sin^2 \frac{\varphi^\prime}{2}}{(|p-p^\prime|^2+4yy^\prime \sin^2 \frac{\varphi^\prime}{2})^{3/2}}
    \e^{\iu k \varphi^\prime} .
\end{align*}
Now we define
\begin{equation}\label{eq_kk}
    K_k^*(p, p^\prime):=\frac{(y y^\prime)^{1/2}}{4\pi} \int_{S^1}
    \frac{(p-p^\prime)\cdot {\widetilde\nv}_p-2v_p y^\prime\sin^2 \frac{\varphi^\prime}{2}}{(|p-p^\prime|^2+4yy^\prime \sin^2 \frac{\varphi^\prime}{2})^{3/2}}
    \e^{\iu k \varphi^\prime}\, \df \varphi^\prime
\end{equation}
so that
\[
    \Ucal\NP^* \Ucal^{-1} [f\otimes e_k](p, \varphi)
    =\e^{\iu k\varphi} \int_{\partial \Sigma} K_k^* (p, p^\prime) f(p^\prime)\, \df \sigma (p^\prime).
\]

It remains to prove \eqref{eq_ak_representation}. Since $\sin^2 (\varphi^\prime/2)$ is an even function, we can reduce \eqref{eq_kk} to
\begin{equation}\label{eq_kk_reduced}
    K^*_k(p, p^\prime)
    =\frac{(yy^\prime)^{1/2}}{\pi}
    \int_0^{\pi/2} \cos (2k\varphi^\prime)\frac{(p-p^\prime)\cdot {\widetilde\nv}_{p}-2v_p y^\prime\sin^2 \varphi^\prime}{(|p-p^\prime|^2+4yy^\prime\sin^2 \varphi^\prime)^{3/2}}\, \df \varphi^\prime.
\end{equation}
We abbreviate $\delta (p, p^\prime)$ to $\delta$. Since
\begin{align*}
    (|p-p^\prime|^2+4yy^\prime\sin^2 \varphi^\prime)^{3/2}
    &=(4yy^\prime)^{3/2}(\delta^2+\sin^2 \varphi^\prime)^{3/2} \\
    &=2\delta^{-2}(yy^\prime)^{1/2}|p-p^\prime|^2 (\delta^2+\sin^2 \varphi^\prime)^{3/2},
\end{align*}
it follows from \eqref{eq_kk_reduced} that
\begin{align*}
    K_k^*(p, p^\prime)
    &=\frac{(yy^\prime)^{1/2}}{\pi}\biggl(
    \int_0^{\pi/2} \cos (2k\varphi^\prime)\frac{\delta^2 (p-p^\prime)\cdot {\widetilde\nv}_p}{2(yy^\prime)^{1/2}|p-p^\prime|^2 (\delta^2+\sin^2 \varphi^\prime)^{3/2}}\, \df \varphi^\prime \\
    &\quad -\int_0^{\pi/2} \cos (2k\varphi^\prime)\frac{2v_p y^\prime\sin^2 \varphi^\prime}{(4yy^\prime)^{3/2}(\delta^2+\sin^2 \varphi^\prime)^{3/2}}\, \df \varphi^\prime \biggr) \\
    &=\frac{\delta^2}{2\pi}\frac{(p-p^\prime)\cdot {\widetilde\nv}_p}{|p-p^\prime|^2}\int_0^{\pi/2} \frac{\cos (2k\varphi^\prime)}{(\delta^2+\sin^2 \varphi^\prime)^{3/2}}\, \df \varphi^\prime \\
    &\quad -\frac{v_p}{4\pi y}\int_0^{\pi/2} \frac{\cos (2k\varphi^\prime)\sin^2 \varphi^\prime}{(\delta^2+\sin^2 \varphi^\prime)^{3/2}}\, \df \varphi^\prime \\
    &=K^*_{\partial\Sigma}(p, p^\prime)A_k(\delta)-\frac{v_p}{4\pi y}B_k(\delta),
\end{align*}
which is the desired formula.

The relation \eqref{eq_np_separated} is an immediate consequence of the duality of $\NP$ and $\NP^*$.
\end{proof}

One can easily see from \eqref{Cks} and Proposition \ref{prop_nps_fourier} that the following corollary holds. 

\begin{coro}
    \label{coro_nps_mode_bounded}
    For $k\in \Zbb$, the operator $\Kcal_k^*$ is a bounded operator on $L^2(\partial\Sigma)$ and on $H^{-1/2}(\partial\Sigma)$.
\end{coro}

\subsection{Decomposition of the single layer potential}

Similarly to the NP operator $\NP^*$, the single layer potential ${\SL}$ is decomposed into Fourier modes.

\begin{prop}
    \label{prop_single_layer_separation}
    For $f\in H^{-1/2} (\partial\Sigma)$ and $k\in \Zbb$, the identity
    \[
        \Ucal {\SL} \Ucal^{-1}[f\otimes e_k]=\Scal_k [f]\otimes e_k
    \]
    holds where
    \begin{equation}
        \label{eq_single_layer_separation_kernel}
            \Scal_k [f](p)=\int_{\partial\Sigma} S_k (p, p^\prime)f(p^\prime)\, \df \sigma (p^\prime)
    \end{equation}
    with
    \[
    S_k(p, p^\prime):=-\frac{1}{2\pi} \int_0^{\pi/2} \frac{\cos (2k\varphi^\prime)}{(\delta (p, p^\prime)^2+\sin^2 \varphi^\prime)^{1/2}}\, \df \varphi^\prime.
    \]
\end{prop}

\begin{proof}
    We see from \eqref{2-1000} that
    \begin{align*}
        |\Psi (p, \varphi)-\Psi (p^\prime, \varphi^\prime)|^2
        &=2yy^\prime \left(\delta (p, p^\prime)^2+\sin^2 \frac{\varphi-\varphi^\prime}{2}\right).
    \end{align*}
    Thus we have
    \begin{align*}
        &\Ucal{\SL}\Ucal^{-1}[f\otimes e_k](p, \varphi) \\
        &=-\frac{1}{4\pi}\int_{\partial\Sigma} \df \sigma (p^\prime)\int_{S^1} \df \varphi^\prime \, \frac{1}{2(yy^\prime)^{1/2}}\frac{(yy^\prime)^{1/2}f(p^\prime)\e^{\iu k\varphi^\prime}}{(\delta(p, p^\prime)^2+ \sin^2 ((\varphi-\varphi^\prime)/2))^{1/2}} \\
        &=-\frac{1}{2\pi} \e^{\iu k \varphi}\int_{\partial\Sigma} \df \sigma (p^\prime)\, f (p^\prime)\int_0^{\pi/2} \df \varphi^\prime \, \frac{\cos (2k \varphi^\prime)}{(\delta (p, p^\prime)^2+ \sin^2 \varphi^\prime)^{1/2}} \\
        &=\e^{\iu k \varphi}\Scal_k [f](p),
    \end{align*}
as desired.
\end{proof}

For $k\in \Zbb$, we set
\begin{equation}
    \label{eq_inner_product_mode}
    \jbk{f, g}_k:=-\jbk{\Scal_k[f], g}_{\partial\Sigma},
\end{equation}
where $\jbk{\cdot, \cdot}_{\partial\Sigma}$ is the dual pairing of $H^{1/2}(\partial\Sigma)$ and $H^{-1/2}(\partial\Sigma)$. Then, Proposition \ref{prop_single_layer_separation} together with \eqref{Cks} immediately yields the following corollary.

\begin{coro}\phantomsection
    \label{coro_inner_product_decomposition}
    \begin{enumerate}[label=(\roman*)]
        \item \label{enum_single_layer_mode_invertible}For $k\in \Zbb$, the operator $\Scal_k$ is an invertible operator from $H^{-1/2}(\partial\Sigma)$ onto $H^{1/2}(\partial\Sigma)$.
        \item \label{enum_inner_product_mode}For $f, g\in H^{-1/2} (\partial\Sigma)$ and $k\in \Zbb$, we have
        \begin{equation}\label{eq_mode_unitary}
            \jbk{f, g}_k=\jbk{\Ucal^{-1}[f\otimes e_k], \Ucal^{-1}[g\otimes e_k]}_*.
        \end{equation}
    \end{enumerate}
\noindent In particular, $\jbk{\cdot, \cdot}_k$ defines an inner product on $H^{-1/2}(\partial\Sigma)$ which induces the norm equivalent to the Sobolev norm on $H^{-1/2}(\partial\Sigma)$.
\end{coro}

Combining Proposition \ref{prop_nps_fourier}, Proposition \ref{prop_single_layer_separation} and the Plemelj's symmetrization principle \eqref{eq_Plemelj}, we have the following proposition:

\begin{prop}
    \label{theo_Plemelj_zero}
    If $\partial\Sigma$ is Lipschitz, then we have
    \begin{equation}\label{Plemelj_k}
        \Scal_k \Kcal_k^*=\Kcal_k \Scal_k
    \end{equation}
    on $H^{-1/2}(\partial\Sigma)$ for each $k \in \Zbb$. In particular, $\Kcal_k^*$ is a bounded self-adjoint operator on $H^{-1/2}(\partial\Sigma)$ equipped with the inner product $\jbk{ \cdot ,  \cdot }_k$. Furthermore, if $\partial\Sigma$ is $C^{1, \alpha}$ for some $\alpha>0$, then the operator $\Kcal_k^*$ is compact on $H^{-1/2}(\partial\Sigma)$.
\end{prop}

\subsection{Proof of Theorem \ref{theo_self_adjoint_zeroth}}

By Proposition \ref{theo_Plemelj_zero}, the operator $\Kcal_0^*$ is a compact self-adjoint operator on $H^{-1/2}(\partial\Sigma)$ with the inner product $\jbk{\cdot, \cdot}_0$. Thus $\Kcal_0^*$ has infinitely many real eigenvalues. Let $\{f_j\}_{j=1}^\infty$ be an orthonormal system of eigenfunctions of $\Kcal_0^*$. Then the function $\{\Ucal^{-1}[f_j\otimes 1]\}_{j=1}^\infty$ forms an orthogonal system of axially symmetric NP eigenfunctions in $H^{-1/2}(\partial\Omega)$ by Proposition \ref{prop_nps_fourier}. \hspace{\fill}\qed

\section{The zeroth mode NP operator}\label{sec:zeroth}

In this section and sections to follow, we deal with the spectral properties of the zeroth mode operator $\NP[0]^*$ appearing in Theorem \ref{prop_nps_fourier}. It is the operator $\Kcal_{\partial\Omega}^*$ restricted to the axially symmetric functions. In this section we show that the integral kernels of the zeroth modes of the NP operator and the single layer potential are represented in terms of elliptic integrals. We use these representations to  investigate the singularity of the integral kernels.

\subsection{Representation by elliptic integrals}

We set $k=0$ in \eqref{eq_AB} to see
    \begin{equation}\label{eq_AB0}
    \begin{aligned}
        A_0 (\delta) &=\delta^2 \int_0^{\pi/2} \frac{\df \varphi}{(\delta^2+\sin^2 \varphi)^{3/2}}, \\
        B_0(\delta) &=\int_0^{\pi/2}\frac{\sin^2 \varphi}{(\delta^2+\sin^2 \varphi)^{3/2}}\, \df \varphi,
    \end{aligned}
    \end{equation}
and
\begin{equation}
        \label{eq_a0_representation}
        K^*_0(p, p^\prime)=K^*_{\partial\Sigma}(p, p^\prime)A_0(\delta (p, p^\prime))-\frac{v_p}{4\pi y}B_0(\delta (p, p^\prime)).
    \end{equation}

We relate function $A_0$ and $B_0$ with the classical elliptic integrals. Let $\Kell (z)$ and $\Ell (z)$ be the complete elliptic integrals of the first and the second kinds, namely,
\begin{equation}\label{eq_elliptic_first}
    \Kell (z):=\int_0^{\pi/2} \frac{\df \varphi}{\sqrt{1-z^2 \sin^2 \varphi}}
\end{equation}
and
\begin{equation}\label{eq_elliptic_second}
    \Ell (z):=\int_0^{\pi/2}\sqrt{1-z^2 \sin^2 \varphi}\, \df \varphi
\end{equation}
for $z\in \Cbb \setminus ((-\infty, -1]\cup [1, \infty))$.

\begin{lemm}\label{prop_AB_formula}
    The following relations hold:
    \begin{equation}
        \label{eq_A0B0}
        \begin{aligned}
            A_0(\delta)&=\frac{\delta \Ell (\iu \delta^{-1})}{1+\delta^2}, \\
            B_0(\delta)&=\delta^{-1} \Kell (\iu \delta^{-1})-\frac{\delta \Ell (\iu \delta^{-1})}{1+\delta^2}
        \end{aligned}
    \end{equation}
    for $\delta> 0$.
\end{lemm}

\begin{proof}

    We easily obtain the following two relations:
    \[
        \frac{1}{(\delta^2+\sin^2 \varphi)^{3/2}}=\frac{1}{\delta (1+\delta^2)}\sqrt{1+\delta^{-2} \sin^2 \varphi}-\frac{1}{2(1+\delta^2)}\frac{\df}{\df \varphi}\frac{\sin (2\varphi)}{\sqrt{1+\delta^{-2} \sin^2 \varphi}}
    \]
    and
    \[
        \begin{aligned}
            &\frac{\sin^2 \varphi}{(\delta^2+\sin^2 \varphi)^{3/2}} \\
            &=\frac{1}{\delta \sqrt{1+\delta^{-2} \sin^2 \varphi}}
            -\frac{\delta}{1+\delta^2}\sqrt{1+\delta^{-2} \sin^2 \varphi}
            -\frac{1}{2\delta (1+\delta^2)}\frac{\df}{\df \varphi}\frac{\sin (2\varphi)}{\sqrt{1+\delta^{-2} \sin^2 \varphi}}
        \end{aligned}
    \]
    for $\delta> 0$. Integrating these relations in $\varphi\in [0, \pi/2]$, we obtain \eqref{eq_A0B0}.
\end{proof}

In order to investigate the asymptotic behavior of $A_0(\delta)$ and $B_0(\delta)$, we obtain the following proposition whose proof is presented in Appendix \ref{sec_proof_elliptic_behavior}.

\begin{prop}
    \label{theo_ell_int_ae}
    The following identities hold for $\delta \in (0,1)$:
    \begin{align}
        \Kell (\iu\delta^{-1})&=-\frac{2\delta \log \delta}{\pi} \Kell (\iu \delta)+\frac{4\delta\log 2}{\pi}+\delta^3 f(\delta), \label{eq_full_expansion_first} \\
        \Ell (\iu\delta^{-1})&=\frac{2(\Ell (\iu \delta)-(1+\delta^2)\Kell (\iu \delta))}{\pi \delta}\log \delta+\frac{1}{\delta}+\delta g(\delta), \label{eq_full_expansion_second}
    \end{align}
    where $f(z)$ and $g(z)$ are holomorphic functions in $|z|<1$.
\end{prop}

Let us look into the singularity of the integral kernel $K_0^*(p, p^\prime)$ roughly. When $\delta=\delta (p, p^\prime)<1$, we obtain the following formulas by substituting \eqref{eq_full_expansion_first} and \eqref{eq_full_expansion_second} to \eqref{eq_A0B0}:
\[
    \begin{aligned}
        A_0(\delta)&=1+\frac{2(\Ell (\iu \delta)-(1+\delta^2)\Kell (\iu \delta))}{\pi (1+\delta^2)}\log \delta+\frac{\delta^2(g(\delta)-1)}{1+\delta^2}, \\
        B_0(\delta)&=-\frac{2\Ell (\iu \delta)}{\pi (1+\delta^2)}\log \delta
        +\frac{4\log 2}{\pi}-1+\delta^2 \left(f(\delta)-\frac{g(\delta)-1}{1+\delta^2}\right).
    \end{aligned}
\]
In particular, we have the asymptotic behaviors
\begin{equation}\label{eq_A0B0_principal}
    A_0 (\delta)=1+O(\delta^2 \log \delta), \quad
    B_0(\delta)=-\log \delta+O(1)
\end{equation}
as $\delta\to 0$. We then infer from \eqref{eq_delta} and \eqref{eq_ak_representation} that
\begin{equation}\label{eq_kernel_singularity_rough}
    K_0^* (p, p^\prime)=K_{\partial\Sigma}^* (p, p^\prime)+\frac{v_p}{4\pi y} \log |p-p^\prime|+O(1)
\end{equation}
as $|p-p^\prime|\to 0$.

Inspired by \eqref{eq_kernel_singularity_rough}, we set
\begin{equation}
    \label{eq_a0_asymptotic_c1a}
    R^* (p, p^\prime):=K^*_0(p, p^\prime)-K_{\partial\Sigma}^*(p, p^\prime)-\frac{v_p}{4\pi y} \log |p-p^\prime|,
\end{equation}
and define the operator $\Rcal^*$ by
\begin{equation}
    \label{eq_rs_operator}
    \Rcal^*[f](p):=\int_{\partial\Sigma} R^*(p, p^\prime) f(p^\prime)\, \df \sigma (p^\prime).
\end{equation}
Then, the following relation holds:
\begin{equation}
    \label{eq_operator_decomposition}
    \NP[0]^*=\NP[\partial\Sigma]^*+\frac{v_p}{2y}\SL[\partial\Sigma]+\Rcal^*
\end{equation}
where $\SL[\partial\Sigma]$ is the single layer potential on $\partial\Sigma$ (see \eqref{eq_single_layer}). We shall investigate $R^* (p, p^\prime)$ in a precise way depending on the regularity of $\p \Sigma$.

\subsection{Regularity of the remainder term}

We now investigate regularity properties of $R^*(p, p^\prime)$. We begin by proving that it is bounded.

\begin{lemm}
    \label{theo_Q_bounded_Linfty}
    If $\partial\Sigma$ is Lipschitz, then $R^* \in L^\infty (\partial\Sigma \times \partial\Sigma)$.
\end{lemm}

\begin{proof}
    By \eqref{eq_A0B0_principal}, we have
    \begin{align*}
        R^*(p, p^\prime)
        =K_{\partial\Sigma}^*(p, p^\prime) \times O(\delta^2 \log \delta)-\frac{v_p}{4\pi y}(-\log \delta+O(1)+\log |p-p^\prime|).
    \end{align*}
    Since $K_{\partial\Sigma}^*(p, p^\prime)=O(|p-p^\prime|^{-1})= O(\delta^{-1})$ as $\delta\to 0$ and $-\log \delta+\log |p-p^\prime|=\log 2(yy^\prime)^{1/2}$, we infer that
    \begin{align*}
        R^*(p, p^\prime)=O(1)
    \end{align*}
    as $\delta=\delta (p, p^\prime)\to 0$, as desired.
\end{proof}

We will use the following lemma to prove $C^{0,1}$ regularity of $\Rcal^*$ (Proposition \ref{theo_rs_L1e_lipschitz}).

\begin{lemm}
    \label{theo_rs_sobolev_estimate}
    If $\partial\Sigma$ is Lipschitz, then
    \begin{equation}
        \label{eq_rs_sobolev_estimate}
        \frac{|R^*(p, p^\prime)-R^*(p, q^\prime)|}{|p^\prime-q^\prime|}\lesssim 1+|\log |p-p^\prime||+|\log |p-q^\prime||
    \end{equation}
    for all distinct points $p, q, q^\prime\in \partial\Sigma$.
\end{lemm}

\begin{proof}
    By \eqref{eq_ak_representation} and \eqref{eq_a0_asymptotic_c1a}, we have
\begin{align*}
    R^*(p, p^\prime)&=K_{\partial\Sigma}^*(p, p^\prime)(A_0 (\delta (p, p^\prime))-1)
    -\frac{v_p}{4\pi y}(B_0(\delta (p, p^\prime))+\log |p-p^\prime|) \\
    &=: Q_1 (p, p^\prime)+Q_2(p, p^\prime).
\end{align*}

It suffices to prove the estimates 
\begin{align}
    \frac{|Q_1 (p, p^\prime)-Q_1 (p, q^\prime)|}{|p^\prime -q^\prime|}
    &\lesssim 1+|\log |p-p^\prime||+|\log |p-q^\prime||,\label{eq_Q1_final} \\
    \frac{|Q_2 (p, p^\prime)-Q_2 (p, q^\prime)|}{|p^\prime-q^\prime|}&\lesssim 1 \label{eq_Q2_final}
\end{align}
for distinct points $p, p^\prime, q^\prime \in \partial \Sigma$. 

First we prove \eqref{eq_Q1_final}. We set
    \[
        Q_{11}(p, p^\prime):=\delta (p, p^\prime)^2 K_{\partial\Sigma}^*(p, p^\prime)
        =-\frac{(p-p^\prime)\cdot \widetilde{\nv}_p}{8\pi yy^\prime}
    \]
    and
    \[
        Q_{12}(p, p^\prime):=\delta (p, p^\prime)^{-2}(A_0 (\delta (p, p^\prime))-1)
    \]
    so that
    \[
        Q_1(p, p^\prime) =Q_{11}(p, p^\prime)Q_{12}(p, p^\prime).
    \]
    Then we have
    \begin{align*}
        &\frac{|Q_1 (p, p^\prime)-Q_1 (p, q^\prime)|}{|p^\prime -q^\prime|} \\
        &\leq \frac{|Q_{11}(p, p^\prime)-Q_{11}(p, q^\prime)||Q_{12} (p, p^\prime)|}{|p^\prime-q^\prime|}
        +\frac{|Q_{11}(p, q^\prime)||Q_{12} (p, p^\prime)-Q_{12} (p, q^\prime)|}{|p^\prime-q^\prime|} \\
        &=: J_1+J_2.
    \end{align*}

    We claim that the estimates 
    \begin{align}
        J_1&\lesssim 1+|\log |p-p^\prime||, \label{eq_J1}\\ 
        J_2&\lesssim 1+|\log |p-p^\prime||+|\log |p-q^\prime|| \label{eq_J2}
    \end{align}
    hold for all distinct points $p, p^\prime, q^\prime\in \partial\Sigma$, which immediately yield \eqref{eq_Q1_final}. We easily obtain
    \[
        |Q_{11}(p, p^\prime)-Q_{11}(p, q^\prime)|\lesssim |p^\prime-q^\prime|.
    \]
    By \eqref{eq_A0B0_principal}, we have
    \[
        |Q_{12} (p, p^\prime)|\lesssim |\log \delta (p, p^\prime)|\lesssim 1+|\log |p-p^\prime||.
    \]
    Combining these estimates, we obtain \eqref{eq_J1}. Next we prove \eqref{eq_J2}. We easily obtain
    \[
        |Q_{11}(p, q^\prime)|\lesssim |p-q^\prime|
    \]
    by the definition of $Q_{11}$. Since
    \[
        \frac{\df}{\df \delta}\frac{A_0 (\delta)-1}{\delta^2}=-\frac{2(1+2\delta^2)}{\delta^2 (1+\delta^2)^2}\Ell (\iu \delta^{-1})+\frac{\Kell (\iu \delta^{-1})}{\delta^2 (1+\delta^2)}+\frac{2}{\delta^3}=O(\delta^{-1})
    \]
    as $\delta\to 0$ by Lemma \ref{prop_AB_formula} and Proposition \ref{theo_ell_int_ae}, we have
    \[
        |Q_{12}(p, p^\prime)-Q_{12}(p, q^\prime)|
        \lesssim \left| \log \frac{\delta (p, p^\prime)}{\delta (p, q^\prime)}\right|.
    \]
    Thus $J_2$ is estimated as
    \[
        J_2\lesssim \frac{\delta (p, q^\prime)}{\delta (p^\prime, q^\prime)} \left|\log \frac{\delta (p, p^\prime)}{\delta (p, q^\prime)}\right|.
    \]
    It follows from Lemma \ref{lemm_triangular_like} that 
    \begin{align*}
        J_2&\lesssim \frac{\delta (p, p^\prime)+\delta (p^\prime, q^\prime)}{\delta (p^\prime, q^\prime)} (|\log \delta (p, p^\prime)|+|\log \delta (p, q^\prime)|) \\
        &\lesssim |\log \delta (p, p^\prime)|+|\log \delta (p, q^\prime)|
    \end{align*}
    if $\delta (p, p^\prime)<\delta (p^\prime, q^\prime)$, and that 
    \begin{align*}
        J_2&\lesssim \frac{\delta (p, q^\prime)}{\delta (p^\prime, q^\prime)} \left|\log \frac{\delta (p, p^\prime)}{\delta (p, q^\prime)}\right|
        \leq \frac{\delta (p, q^\prime)}{\delta (p^\prime, q^\prime)}\frac{|\delta (p, p^\prime)-\delta (p, q^\prime)|}{\min \{ \delta (p, p^\prime), \delta (p, q^\prime)\}} \\
        &\lesssim \delta (p, q^\prime)\left( \frac{1}{\delta (p, p^\prime)}+\frac{1}{\delta (p, q^\prime)}\right) \\
        &=\frac{\delta (p, q^\prime)}{\delta (p, p^\prime)} +1
        \lesssim \frac{\delta (p, p^\prime)+\delta (p^\prime, q^\prime)}{\delta (p, p^\prime)}+1\leq 3
    \end{align*}
    if $\delta (p, p^\prime)\geq \delta (p^\prime, q^\prime)$. Thus we have
    \begin{equation*}
        J_2\lesssim |\log \delta (p, p^\prime)|+|\log \delta (p, q^\prime)|+1
        \lesssim 1+|\log |p-p^\prime||+|\log |p-q^\prime||
    \end{equation*}
    for all distinct points $p, p^\prime, q^\prime\in \partial\Sigma$. This completes the proof of \eqref{eq_J2}.

    As the final step, we prove the estimate \eqref{eq_Q2_final}. Since
    \[
        Q_2(p, p^\prime)=B_0 (\delta (p, p^\prime))+\log \delta (p, p^\prime)+\frac{1}{2}\log (4yy^\prime),
    \]
    we have
    \begin{align*}
        &\frac{|Q_2 (p, p^\prime)-Q_2 (p, q^\prime)|}{|p^\prime -q^\prime|} \\
        &\lesssim 1+\frac{|B_0 (\delta (p, p^\prime))+\log \delta (p, p^\prime)-B_0 (\delta (p, q^\prime))-\log \delta (p, q^\prime)|}{|p^\prime -q^\prime|}.
    \end{align*}
    By Lemma \ref{prop_AB_formula} and Proposition \ref{theo_ell_int_ae}, we have
    \begin{align*}
        \frac{\df}{\df \delta}(B_0 (\delta)+\log \delta)&=-\frac{1-\delta^2}{(1+\delta^2)^2}\Ell (\iu \delta^{-1})-\frac{\Kell (\iu \delta^{-1})}{1+\delta^2}+\frac{1}{\delta} \\
        &=O(\delta \log \delta)=O(1)
    \end{align*}
    as $\delta \to 0$, and hence
    \begin{equation*}
        \frac{|Q_2 (p, p^\prime)-Q_2 (p, q^\prime)|}{|p^\prime-q^\prime|}
        \lesssim 1+\frac{|\delta (p, p^\prime)-\delta (p, q^\prime)|}{|p^\prime-q^\prime|}\lesssim 1
    \end{equation*}
    by Lemma \ref{lemm_triangular_like}. This completes the proof of \eqref{eq_Q2_final}. 
\end{proof}

We consider the $L^2$-adjoint $\Rcal$ of $\Rcal^*$:
\[
    \Rcal [f](p):=\int_{\partial\Sigma} R(p, p^\prime)f(p^\prime)\, \df \sigma (p^\prime)
\]
where
\[
    R(p, p^\prime):=R^*(p^\prime, p)=K_0 (p, p^\prime)-K_{\partial\Sigma}(p, p^\prime)-\frac{v_{p^\prime}}{4\pi y^\prime} \log |p-p^\prime|.
\]
By Lemma \ref{theo_Q_bounded_Linfty} and \ref{theo_rs_sobolev_estimate}, we obtain the following proposition.
\begin{prop}\label{theo_rs_L1e_lipschitz}
    For any $\varepsilon>0$, the operator $\Rcal$ is bounded as acting from $L^{1+\varepsilon}(\partial\Sigma)$ into $C^{0, 1}(\partial\Sigma)$.
\end{prop}

\begin{proof}
    By Lemma \ref{theo_Q_bounded_Linfty}, the operator $\Rcal$ is bounded as acting from $L^1 (\partial\Sigma)$ into $L^\infty (\partial\Sigma)$. By Lemma \ref{theo_rs_sobolev_estimate} and the H\"older inequality, we have
    \begin{align*}
        \frac{|\Rcal[f](p)-\Rcal[f](q)|}{|p-q|}
        &\lesssim \int_{\partial\Sigma} (1+|\log |p-p^\prime||+|\log |q-p^\prime||)|f(p^\prime)|\, \df \sigma (p^\prime) \\
        &\lesssim \|f\|_{L^{1+\varepsilon}(\partial\Sigma)}
    \end{align*}
    for all $\varepsilon>0$. Thus,
    we have
    \[
        \| \Rcal[f] \|_{C^{0, 1}(\partial\Sigma)}=\| \Rcal[f] \|_{L^\infty (\partial\Sigma)}
        +\sup_{\substack{p, q\in \partial\Sigma \\ p\neq q}}\frac{|\Rcal[f](p)-\Rcal[f](q)|}{|p-q|} \lesssim \|f\|_{L^{1+\varepsilon}(\partial\Sigma)},
    \]
    which is the desired estimate.
\end{proof}

We recall a characterization of the Sobolev space. Since $\dim \partial\Sigma=1$, the norm on the Sobolev space $H^\nu (\partial\Sigma)$ for $\nu\in (0, 1)$ is equivalent to the norm
\begin{equation}
    \label{eq_sobolev_characterization}
    \|f\|_{H^\nu}:=\left(\| f\|_{L^2(\partial\Sigma)}^2+\int_{\partial\Sigma \times \partial\Sigma} \frac{|f(p)-f(q)|^2}{|p-q|^{2\nu+1}}\, \df \sigma (p)\df \sigma (q)\right)^{1/2}.
\end{equation}
(See \cite{Gilbarg-Trudinger01} for example.) Then, we can observe that the space $C^{0, 1}(\partial\Sigma)$ is continuously embedded in the Sobolev space $H^{1-\varepsilon}(\partial\Sigma)$ for all $\varepsilon>0$. Thus, we can prove the following corollary:
\begin{coro}
    \label{coro_rs_lipschitz}
    If $\partial \Sigma$ is Lipschitz, then the operators $\Rcal$ and $\Rcal^*$ are compact operators on $H^{1/2}(\partial\Sigma)$ and $H^{-1/2}(\partial\Sigma)$ respectively.
\end{coro}

\begin{proof}
    Since the embedding $C^{0, 1}(\partial\Sigma)\hookrightarrow H^{1/2}(\partial\Sigma)$ is continuous, we infer from Proposition \ref{theo_rs_L1e_lipschitz} that the operator $\Rcal: L^2 (\partial\Sigma)\to H^{1/2}(\partial\Sigma)$ is bounded. Since the embedding $H^{1/2}(\partial\Sigma) \hookrightarrow L^2 (\partial\Sigma)$ is compact by the Rellich-Kondrachov theorem, the operator $\Rcal$ is compact on $H^{1/2}(\partial\Sigma)$. Then, by the duality, the operator $\Rcal^*$ is compact on $H^{-1/2}(\partial\Sigma)$.
\end{proof}

We will also employ the Sobolev space of functions with two variables. For $\mu, \nu\geq 0$, we denote the Sobolev space for functions with two variables on $\partial\Sigma$ of order $\mu$ and $\nu$ in the first and second variable respectively by $H^{\mu, \nu}(\partial \Sigma \times \partial \Sigma)$. For more details of definition, see \cite{Delgado-Ruzhansky14}. In this paper, we only employ the case when $\mu=0$. In this case, $H^{0, \nu}(\partial \Sigma \times \partial \Sigma)=L^2_p (\partial\Sigma, H^\nu_{p^\prime} (\partial\Sigma))$ and any function $A(p, p^\prime)$ in $H^{0, \nu}(\partial\Sigma \times \partial \Sigma)$ is characterized by the finiteness of the following quantity:  
\begin{equation}
    \label{eq_sobolev_double}
    \begin{aligned}
        &\left(\int_{\partial\Sigma} \| A(p, \cdot)\|_{H^\nu}^2 \,\df \sigma (p)\right)^{1/2} \\
        &=\left( \|A\|_{L^2 (\partial\Sigma\times \partial\Sigma)}^2+\int_{(\partial\Sigma)^3} \frac{|A(p, p^\prime)-A(p, q^\prime)|^2}{|p^\prime-q^\prime|^{2\nu+1}}\,\df \sigma (p)\df \sigma (p^\prime)\df \sigma (q^\prime)\right)^{1/2}. 
    \end{aligned}
\end{equation}

The following proposition plays a crucial role in investigating decay rate of eigenvalues.

\begin{prop}\label{theo_rs_sobolev_kernel}
    If $\partial\Sigma$ is Lipschitz, then $R^*(p, p^\prime)\in H^{0, 1-\varepsilon}(\partial\Sigma\times \partial\Sigma)$ for all $\varepsilon>0$.
\end{prop}

\begin{proof}
    By Lemma \ref{theo_Q_bounded_Linfty}, we have $R^*\in L^2(\partial\Sigma\times \partial\Sigma)$. By Lemma \ref{theo_rs_sobolev_estimate}, we have
    \begin{align*}
        &\int_{(\partial\Sigma)^3} \frac{|R^* (p, p^\prime)-R^* (p, q^\prime)|^2}{|p^\prime-q^\prime|^{3-2\varepsilon}}\,\df \sigma (p)\df \sigma (p^\prime)\df \sigma (q^\prime) \\
        &\lesssim \int_{(\partial\Sigma)^3} \frac{(1+|\log |p-p^\prime||+|\log |p-q^\prime||)^2}{|p^\prime-q^\prime|^{1-2\varepsilon}}\,\df \sigma (p)\df \sigma (p^\prime)\df \sigma (q^\prime)<\infty.
    \end{align*}
    Thus we have $R^*\in H^{0, 1-\varepsilon}(\partial\Sigma \times \partial\Sigma)$.
\end{proof}

\section{Proof of Theorem \ref{theo_eigenvalue_C1a_C2a}}

\subsection{Schatten class}\label{subs_schatten}

For a separable Hilbert space $\Hcal$, we denote the $p$-th Schatten class on $\Hcal$ by $\Scal^p (\Hcal)$. More explicitly, if we denote the $j$-th singular value of a compact linear operator $\Acal$ by $s_j=s_j (\Acal)=s_j (\Acal, \Hcal)$ in the descending order
\[
    s_1(\Acal)\geq s_2(\Acal)\geq s_3(\Acal)\geq \cdots \,(\to 0),
\]
then the $p$-th Schatten class $\Scal^p (\Hcal)$ is the collection of all compact operators such that
\[
    \sum_{j=1}^\infty s_j (\Acal)^p <\infty.
\]
In particular, the class $\Scal^2 (\Hcal)$ is called the Hilbert-Schmidt class and $\Scal^1 (\Hcal)$ is called the trace class. 

We will work on the singular values of $\NP[0]^*$ on $L^2 (\partial\Sigma)$, not on $H^{-1/2}(\partial \Sigma)$. We will connect the singular value asymptotics on $L^2 (\partial\Sigma)$ and the eigenvalue asymptotics on $H^{-1/2}(\partial\Sigma)$ as follows. First, we derive an eigenvalue asymptotics on $L^2 (\partial\Sigma)$ from the singular value asymptotics on the same space (Lemma \ref{lemm_schatten_decay}). Then, we employ the equivalence between the eigenvalue problems for $\NP^*$ on $H^{-1/2}(\partial\Omega)$ and on $L^2 (\partial\Omega)$ when $\partial\Sigma$ is $C^{1, \alpha}$ for some $\alpha>0$, in the sense that every eigenfunction of $\NP^*$ in $H^{-1/2}(\partial\Omega)$ with nonzero eigenvalue actually belongs to $L^2 (\partial\Omega)$ (see \cite[Corollary A.3]{FKMdr} for example). By restricting the axially symmetric eigenfunctions, we obtain the equivalence between the eigenvalue problems of $\NP[0]^*$ on $H^{-1/2}(\partial\Sigma)$ and on $L^2 (\partial\Sigma)$.

Now we recall the following general result of Delgado and Ruzhansky.

\begin{theo}[{\cite[Theorem 3.6]{Delgado-Ruzhansky14}}]\label{theo_DR14}
    Let $M$ be a $C^\infty$-smooth closed manifold of dimension $n$ equipped with a $C^\infty$-smooth density $\df\mu (x)$. (For example, equip $M$ with a Riemannian metric and consider the volume form associated with the Riemannian metric.) If $K(x, y)$ belongs to the Sobolev space $H^{0, s}(M\times M):=L^2_x (M, H^s_y (M))$ for some order $s>0$, then the linear operator
    \[
        \Tcal [f](x):=\int_M K(x, y)f(y)\, \df \mu (y)
    \]
    belongs to the $p$-th Schatten class $\Scal^p (L^2(M))$ with
    \[
        p>\frac{2n}{n+2s}.
    \]
\end{theo}

Here we make a small remark in order to avoid possible confusion (this remark applies to Theorem \ref{theo_grubb} as well). 
The $C^\infty$-smooth manifold does not mean that it is the boundary of a domain with the smooth boundary. It means that $\varphi_\alpha \circ \varphi_\beta^{-1}$ is smooth, where $\varphi_\alpha, \ \varphi_\beta$ are coordinate charts. The Lipschitz continuous boundary in the Euclidean space can be realized as a smooth manifold. In fact, if $\partial \Sigma$ is Lipschitz, then the there is a domain $D$ such that $\partial D$ is smooth and $\partial D$ is homeomorphic to $\partial \Sigma$ (see \cite{Necas62} or \cite[Theorem 1.12]{Verchota84}). We can define the coordinate charts on $\partial \Sigma$ in terms of the smooth coordinate charts on $\partial D$ by using the homeomorphism. The boundary regularity only affects the property of the integral kernel of the NP operator.

Proposition \ref{theo_rs_sobolev_kernel} and Theorem \ref{theo_DR14} immediately yield the following proposition.

\begin{prop}\label{theo_np_zero_mode_decomposition}
    If $\partial\Sigma$ is Lipschitz, then $\Rcal^*\in \Scal^{2/3+\varepsilon}(L^2 (\partial\Sigma))$ for all $\varepsilon>0$.
\end{prop}

We recall the following result on the Schatten class to which the NP operator on the two-dimensional domain belongs depending on the regularity of the boundary.

\begin{prop}[{\cite{FKMdr}}]
    \label{theo_nps_Sobolev}
    If $\Sigma$ is a bounded domain in $\Rbb^{2}$ with $C^{k, \alpha}$ boundary $\partial\Sigma$ for some positive integer $k$ and $\alpha \in (0,1]$, then $\NP[\partial\Sigma]^*\in \Scal^{p} (L^2 (\partial\Sigma))$ for all $p > 1/(k+\alpha-1)$.
\end{prop}

\subsection{Pseudodifferential operators}\label{subs_psiDO}
We now look into the decay rate of the singular values of the operator $(v_p/2y)\SL[\partial\Sigma]$. To do so, we invoke a result of Grubb \cite{Grubb14}. We first prepare ourselves with basic notion of pseudodifferential operators (abbreviated to $\Psi$DOs). In what follows, we use the notation in \cite{Grubb14} except for the space $C^\tau S^m (\Rbb^d\times (\Rbb^d\setminus \{0\}))$, which will be defined later. We further refer to \cite{Hormander85-3, Kumano-go81, Zworski12} for basics of $\Psi$DOs. 

We denote the Schwartz class (of all rapidly decreasing functions) on $\Rbb^d$ by $\Sscr (\Rbb^d)$ and the space of all tempered distributions on $\Rbb^d$ by $\Sscr^\prime (\Rbb^d)$ (continuous dual of $\Sscr (\Rbb^d)$). For $a\in \Sscr^\prime (\Rbb^{2d})$ and $u\in \Sscr (\Rbb^d)$, we define 
\[
    a(x, D)u(x):=\frac{1}{(2\pi)^d}\int_{\Rbb^d}a(x, \xi) \e^{\iu \xi \cdot x}\widehat{u}(\xi)\, \df \xi \in \Sscr^\prime (\Rbb^d)
\]
and call $a(x, D)$ the $\Psi$DO with the symbol $a(x, \xi)$. Here 
\[
    \widehat{u}(\xi):=\int_{\Rbb^d} u(x)\e^{-\iu \xi \cdot x}\, \df x \in \Sscr (\Rbb^d)
\]
is the Fourier transform of $u(x)\in \Sscr (\Rbb^d)$. We say that a tempered distribution $a(x, \xi)$ on $\Rbb^{2d}$ belongs to $C^\tau  S^m_{1, 0} (\Rbb^{2d})$ if $a$ belongs to $C^\infty_\xi (\Rbb^d, C^\tau _x(\Rbb^d))$ and 
\[
    \| \partial_{\xi}^\beta a(\cdot, \xi)\|_{C^\tau  (\Rbb^d)}=O((1+|\xi|)^{m-|\beta|}) \quad (|\xi|\to \infty)
\]
for any $\beta\in (\Nbb\cup \{0\})^d$ where we employed the standard multiindex notation 
\[
    \partial_\xi^\beta :=\partial_{\xi_1}^{\beta_1}\cdots \partial_{\xi_d}^{\beta_d}, \ 
    |\beta|:=\beta_1+\cdots+\beta_d.
\]
We say that $a\in \Sscr^\prime (\Rbb^{2d})$ belongs to $C^\tau  S^m (\Rbb^{2d})$ ($\tau>0$, $m\in \Rbb$) if there exists a family of functions $\{a_{m-j}: \Rbb^{2d}\to \Cbb\}_{j=0}^\infty$ such that $a_{m-j}\in C^\tau  S^{m-j}_{1, 0} (\Rbb^{2d})$ is homogeneous of degree $m-j$ for $\xi$ with $|\xi|\geq 1$ and $a-\sum_{j=0}^J a_{m-j}\in C^\tau  S^{m-J-1}_{1, 0} (\Rbb^{2d})$ for all $J\in \Nbb$. We say that $a\in \Sscr^\prime (\Rbb^{2d})$ belongs to $C^\tau  S^m (\Rbb^d\times (\Rbb^d\setminus \{0\}))$ if $(1-\chi (\xi))a(x, \xi)\in C^\tau  S^m (\Rbb^{2d})$ for any function $\chi \in C_c^\infty (\Rbb^d)$ with $\chi(\xi)=1$ near $\xi=0$. 

Next, we consider $\Psi$DOs on a compact $d$-dimensional manifold $M$. We denote its cotangent bundle by $T^*M$. Take a finite covering $\{U_j\}_{j=1}^N$ of $M$ such that, for any pair $(j, k)$ with $j\neq k$ and $U_j\cap U_k\neq \varnothing$, there exists a coordinate neighborhood $U_{jk}\subset M$ such that $U_j \cap U_k\subset U_{jk}$. Let $\varphi_{jk}: U_{jk}\to V_{jk}$ ($V_{jk}$ is an open subset of $\Rbb^d$) be the coordinate chart and $\widetilde{\varphi}_{jk}: T^* U_{jk}\to V_{jk}\times \Rbb^d$ be the associated canonical coordinates, namely, 
\begin{equation}\label{eq_canonical}
    \widetilde{\varphi}_{jk}\left(x, \sum_{l=1}^d \xi_l\, \df x_l|_x\right):= (\varphi_{jk}(x), \xi_1, \ldots, \xi_d),
\end{equation}
where $\df x_1|_x, \ldots, \df x_d|_x$ be the local basis of $T^*_x U_{jk}$ associated with the coordinate chart $\varphi_{jk}$. 
Fix a partition of unity $\{\kappa_j\}_{j=1}^N$ subordinate to the open covering $\{U_j\}_{j=1}^N$. We say that $a\in \Dscr^\prime (T^*M)$, where $\Dscr^\prime (T^*M)$ is the space of all distributions on $T^*M$ (see \cite{Hormander85-1} for example), is locally in $C^\tau  S^m (\Rbb^d \times (\Rbb^d\setminus \{0\}))$ if we can take an atlas described above such that $\chi (x)\widetilde{\varphi}_{jk*}a(x, \xi)\in C^\tau  S^m (\Rbb^d\times (\Rbb^d\setminus \{0\}))$ for any function $\chi \in C_c^\infty (V_{jk})$ where $\widetilde{\varphi}_{jk*}a(x, \xi):=a(\widetilde{\varphi}_{jk}^{-1}(x, \xi))$ ($(x, \xi)\in V_{jk}\times \Rbb^d$). Then, for a function $a\in \Dscr^\prime (T^*M)$ which is locally in $C^\tau  S^m (\Rbb^d\times (\Rbb^d\setminus \{0\}))$, we define
\begin{equation}\label{eq_psido}
    \Op (a)u(x):=\sum_{\substack{j, k=1 \\ U_j\cap U_k \neq\varnothing}}^N \kappa_j (x)(\widetilde{\varphi}_{jk*}a)(x, D)[\kappa_k (u\circ \varphi_{jk})](x)
\end{equation}
for $u\in C^\infty (M)$. We remark that each summand makes sense since $\kappa_j \tilde{\varphi}_{jk*}a\in C^\tau  S^m (\Rbb^d\times (\Rbb^d\setminus \{0\}))$. Finally, for $\Acal=\Op (a)$ with the symbol $a$ which is locally in $C^\tau  S^m (\Rbb^d\times (\Rbb^d\setminus \{0\}))$, we define the principal symbol $a^\mathrm{pr}\in \Dscr^\prime (T^*M)$ of $\Acal$, which is also locally in $C^\tau S^m (\Rbb^d\times (\Rbb^d\setminus \{0\}))$, by the following procedure: for $(x_0, \xi_0)\in T^*M$ with $\xi\neq 0$, we find a coordinate neighborhood $U_{jk}$ containing $x$. Then take a cutoff functions $\chi_1 (x)\in C_c^\infty (V_{jk})$ and $\chi_2 (\xi)\in C_c^\infty (\Rbb^d)$ such that $\chi_1 (x)=1$ near $x_0$, $\chi_2(\xi)=1$ near $\xi=0$ and $\chi_2 (\xi)=0$ near $\xi=\xi_0$. Then we can take $a_{jk, m}\in C^\tau  S^m (\Rbb^{2d})$ such that $\chi_1 (x)(1-\chi_2 (\xi))\widetilde{\varphi}_{jk*}a (x, \xi)-a_{jk, m}(x, \xi)\in C^\tau  S^{m-1}_{1, 0}(\Rbb^{2d})$. Then we define $a^\mathrm{pr}(x_0, \xi_0):=a_{jk, m}(\widetilde{\varphi}_{jk}^{-1}(x_0, \xi_0))$. This definition of the principal symbol is actually independent of the choice of the coordinate chart $\varphi_{jk}: U_{jk}\to V_{jk}$ and the cutoff functions $\chi_1(x)$ and $\chi_2(\xi)$. We also denote the principal symbol of $\Acal$ by $\sigma^\mathrm{pr}(\Acal)$ when there is no possibility of confusion. 

Now we state the following theorem which is the one-dimensional case of the result in \cite{Grubb14}.

\begin{theo}[{\cite[Theorem 2.5]{Grubb14}}]\label{theo_grubb}
    If $\Acal=\Op (a)$ is a $\Psi$DO of order $-1$ with the symbol $a(x, \xi)$ and the principal symbol $a^\mathrm{pr}(x, \xi)$ which are locally in $C^\tau  S^{-1}(\Rbb\times (\Rbb\setminus \{0\}))$ for some $\tau>0$ on a $C^\infty$-smooth one-dimensional compact manifold $M$ and $g$ be a $C^{0, \alpha}$-Riemannian metric on $M$ for some $\alpha>0$, then the singular values of $\Acal$ on $L^2 (M, \df \sigma)$ where $\df \sigma$ is the measure associated with the Riemannian metric $g$ has the asymptotics
    \[
        s_j (\Acal)\sim C_0j^{-1}, \quad
        C_0=\frac{1}{2\pi}\int_M (|a^\mathrm{pr}(x, \xi (x))|+|a^\mathrm{pr}(x, -\xi (x))|)\, \df \sigma (x),
    \]
    where $\xi: M\to T^*M$ is a continuous global section such that $|\xi(x)|=1$ for all $x\in M$ with respect to the metric $g$.
\end{theo}

If $\partial \Sigma$ is $C^\infty$, it is well-known that $\SL$ itself is a $\Psi$DO and its principal symbol is given by $-1/2|\xi|$ (see \cite{AKSV99} for example). Thus we can directly apply Theorem \ref{theo_grubb} to the operator $(v_p/2y)\SL[\partial\Sigma]$. However, when the boundary is merely $C^{1, \alpha}$ for some $\alpha\in (0, 1)$ as in this section, the operator $\SL[\partial\Sigma]$ may not be a $\Psi$DO. In this case, we decompose the operator $(v_p/2y)\SL[\partial\Sigma]$ into the sum of a $\Psi$DO and a Schatten class operator as in the following proposition: 

\begin{prop}\label{prop_sl_decomposition}
    If $\partial\Sigma$ is $C^{1, \alpha}$ for some $\alpha\in (0, 1)$, then the operator $(v_p/2y)\SL[\partial\Sigma]$ is decomposed as 
    \begin{equation}
        \label{eq_sl_decomposition}
        \frac{v_p}{2y}\SL[\partial\Sigma]=\Op \left(-\frac{v_p}{4y|\xi|}\right)+\Lcal,
    \end{equation}
    where $\Lcal\in \Scal^{2/(1+2\alpha)+\varepsilon}(L^2 (\partial \Sigma))$ for all $\varepsilon>0$. 
\end{prop}

\begin{proof}
    We begin by recalling a well-known formula \cite[p.132]{Vladimirov71}
    \begin{equation}\label{eq_log_Fourier}
        \widehat{\log |t|} (\tau)=-\frac{\pi}{|\tau|}-2\pi \gamma \delta(\tau), 
    \end{equation}
    where $\gamma$ is the Euler-Mascheroni constant and $\delta (\tau)$ is the delta function. Let $U$ be a coordinate neighborhood in $\partial\Sigma$ and $\varphi: U\subset \partial \Sigma \to V\subset \Rbb$ be a coordinate chart whose inverse $\varphi^{-1}: V\to U$ is $C^{1, \alpha}$-regular as a function from $V\subset \Rbb$ into $\Rbb^2$. Let $\kappa_1, \kappa_2: \partial\Sigma \to [0, \infty)$ be cutoff functions supported in $U$ which can be extended as compactly supported $C^\infty$ functions in $\Rbb^2$. Then a direct calculation by the Taylor theorem and the formula \eqref{eq_log_Fourier} proves 
    \begin{equation}
        \label{eq_sl_local}
        \kappa_1 \SL[\partial \Sigma][\kappa_2 f](p)=-\frac{\sqrt{|\Phi (t)|}}{2}\kappa_1 (p)\varphi^*\left(\frac{1}{|D|}[\varphi_*(\kappa_2 f)]\right)(p)+\widetilde{\Lcal}[f](p)
    \end{equation}
    for $f\in C^1 (\partial\Sigma)$, where $\Phi (t):=\df \varphi^{-1}(t)/\df t\in \Rbb^2$, 
    \[
        \left(\frac{1}{|D|}\right) [u](t):=(2\pi)^{-1}\int_\Rbb |\tau|^{-1}\e^{\iu \tau t}\widehat{u}(\tau)\, \df \tau
    \]
    is the $\Psi$DO with the symbol $|\tau|^{-1}\in \Sscr^\prime (\Rbb)$, and the operator $\widetilde{\Lcal}$ is defined by 
    \[
        \widetilde{\Lcal}[f](p):=\frac{\kappa_1 (p)}{2\pi}\int_{\partial\Sigma} L (\varphi (p), \varphi (q)) |\varphi^*\Phi (q)|^{-1}\kappa_2 (q) f(q)\, \df \sigma (q)
    \]
    where  
    \begin{equation}\label{eq_tildeB}
        L (t, s):=(|\Phi (s)|-|\Phi (t)|)\log |s-t|+|\Phi (s)|\log |\Phi (t)+\Xi (t, s)|-2\pi \gamma 
    \end{equation}
    and 
    \[
        \Xi (t, s):=\int_0^1 (\Phi (t+\theta (s-t))-\Phi (t))\, \df \theta \in \Rbb^2. 
    \]
    
    It suffices to prove $\widetilde{\Lcal}\in \Scal^{2/(1+2\alpha)+\varepsilon} (L^2 (\partial\Sigma))$ for all $\varepsilon>0$. In fact, the desired decomposition \eqref{eq_sl_decomposition} and the operator $\Lcal$ there are obtained by taking a finite atlas and a partition of unity as we described in the beginning of this subsection and summing up \eqref{eq_sl_local} over the index of the atlas following the definition \eqref{eq_psido} of the $\Psi$DO. This delocalization procedure and the multiplication by $v_p/2y\in L^\infty (\partial\Sigma)$ does not change the Schatten class property of $\widetilde{\Lcal}$.

    In order to prove $\widetilde{\Lcal}\in \Scal^{2/(1+2\alpha)+\varepsilon}(L^2 (\partial\Sigma))$ for any $\varepsilon>0$, it suffices to prove 
    \[
        L (\varphi (p), \varphi (q))|\varphi^*\Phi (q)|^{-1}\kappa_1 (p)\kappa_2 (q) \in H^{0, \alpha-\varepsilon} (\partial\Sigma \times \partial \Sigma)
    \]
    for all $\varepsilon>0$ by virtue of Theorem \ref{theo_DR14}. It is easy to prove that the second and third terms in the right hand side of \eqref{eq_tildeB} belong to the Sobolev space $H^{0, \alpha-\varepsilon}(\partial\Sigma\times \partial\Sigma)$ for all $\varepsilon>0$ by directly checking the condition \eqref{eq_sobolev_double} using the $C^{1, \alpha}$-regularity of $\varphi^{-1}: V\to U\subset \partial\Sigma\subset \Rbb^2$. Thus it suffices to prove 
    \begin{equation}
        \label{eq_L_Sobolev}
        \int_{I^3} \frac{\left|L_1 (t, s)-L_1 (t, s^\prime)\right|^2}{|s-s^\prime|^{1+2(\alpha-\varepsilon)}}\, \df t \df s \df s^\prime <\infty   
    \end{equation}
    for any $\varepsilon>0$, where 
    \begin{align*}
        &L_1 (t, s):=(|\Phi (t)|-|\Phi (s)|)\log |t-s|, \\
        &I:=\rmop{supp}\varphi_*\kappa_1 \cup \rmop{supp}\varphi_*\kappa_2\subset \Rbb. 
    \end{align*}
        
    We introduce a subset 
    \[
        S:=\{ (t, s, s^\prime)\in I^3 \mid 2|s-s^\prime|<|t-s|\}
    \]
    and decompose the integral \eqref{eq_L_Sobolev} into those over $S$ and over $I^3\setminus S$. 

    Since $|t-s^\prime|\geq |t-s|/2$ on $S$, one can prove the inequality 
    \[
        \left|\log |t-s|-\log |t-s^\prime|\right| \leq \frac{2|s-s^\prime|}{|t-s|}
    \]
    on $S\setminus \{ t=s\}$ by the mean value theorem. Thus, since $|t-s^\prime|\leq 3|t-s|$ on $S$, we have the estimate 
    \begin{align*}
        &\int_S \frac{\left|L_1 (t, s)-L_1 (t, s^\prime)\right|^2}{|s-s^\prime|^{1+2(\alpha-\varepsilon)}}\, \df t \df s \df s^\prime \\
        &\lesssim \int_S \biggl(\frac{|\Phi (s)-\Phi (s^\prime)|^2 |\log |t-s||^2}{|s-s^\prime|^{1+2(\alpha-\varepsilon)}} \\
        &\quad +\frac{|\Phi (t)-\Phi (s^\prime)|^2 |\log |t-s|-\log |t-s^\prime||^2}{|s-s^\prime|^{1+2(\alpha-\varepsilon)}}\biggr)\, \df t \df s \df s^\prime \\ 
        &\lesssim \int_S \left(\frac{|\log |t-s||^2}{|s-s^\prime|^{1-2\varepsilon}}+\frac{1}{|s-s^\prime|^{2(\alpha-\varepsilon)-1}|t-s|^{2-2\alpha}}\right)\, \df t \df s \df s^\prime \\
        &\lesssim \int_{I^2} (|t-s|^{2\varepsilon}|\log |t-s||^2+|t-s|^{2\varepsilon})\, \df t \df s<\infty.
    \end{align*} 
    
    On the other hand, since $|L_1 (t, s)|\lesssim |t-s|^\alpha |\log |t-s||$ and $|t-s^\prime|\leq 3|s-s^\prime|$ on $I^3\setminus S$, we obtain the estimate 
    \begin{align*}
        &\int_{I^3\setminus S} \frac{\left|L_1 (t, s)-L_1 (t, s^\prime)\right|^2}{|s-s^\prime|^{1+2(\alpha-\varepsilon)}}\, \df t \df s \df s^\prime \\
        &\lesssim \int_{I^3\setminus S} \frac{|t-s|^{2\alpha}|\log |t-s||^2}{|s-s^\prime|^{1+2(\alpha-\varepsilon)}}\, \df t \df s \df s^\prime \\
        &\quad +\int_{I^3, |t-s^\prime|\leq 3|s-s^\prime|}\frac{|t-s^\prime|^{2\alpha}|\log |t-s^\prime||^2}{|s-s^\prime|^{1+2(\alpha-\varepsilon)}}\, \df t \df s \df s^\prime \\ 
        &\lesssim \int_{I^2} |t-s|^{2\varepsilon}|\log |t-s||^2\, \df t \df s
        +\int_{I^2}|t-s^\prime|^{2\varepsilon}|\log |t-s^\prime||^2\, \df t \df s^\prime \\
        &<\infty.
    \end{align*}
    This completes the proof of \eqref{eq_L_Sobolev}. 
\end{proof}

\subsection{Proof of Theorem \ref{theo_eigenvalue_C1a_C2a}}\label{subs_singular_value_C1a}

To prove Theorem \ref{theo_eigenvalue_C1a_C2a}, we invoke two well-known results. The first result is a classical one whose proof can be found in \cite[Corollary 3.2 in p.41 and (7.12) in p.95]{Gohberg-Krein69}.

\begin{lemm}
    \label{lemm_schatten_decay}
    Let $p> 0$, $\Acal$ be a compact operator on a Hilbert space $\Hcal$, and $\lambda_j (\Acal)$ be eigenvalues of $\Acal$ in the descending modulus order. 
    \begin{enumerate}[label=(\roman*)]
        \item \label{enum_schatten_decay}If $\Acal$ belongs to the Schatten class $\Scal^p (\Hcal)$, then we have
        \[
            |\lambda_j (\Acal)|=o(j^{-1/p})
        \]
        as $j\to \infty$.
        \item \label{enum_singular_ev_decay}If $\Acal$ has the singular value asymptotics $s_j (\Acal)=O(j^{-p})$ as $j\to \infty$ and it has infinitely many eigenvalues, then 
        \[
            |\lambda_j (\Acal)|=O(j^{-p})
        \]
        as $j\to \infty$. 
    \end{enumerate}
\end{lemm}

The second one is the Ky Fan theorem whose proof can be found in \cite[p. 32]{Gohberg-Krein69}.

\begin{theo}[Ky Fan theorem]
    \label{theo_ky_fan}
    If compact operators $\Acal$ and $\Bcal$ on a same Hilbert space satisfy $s_j (\Acal)\sim Cj^{-p}$ and $s_j (\Bcal)=o(j^{-p})$ for some $C>0$ and $p>0$, then
        \[
            s_j (\Acal+\Bcal)\sim Cj^{-p}.
        \]
\end{theo}

Now we are ready to prove Theorem \ref{theo_eigenvalue_C1a_C2a}. 

\begin{proof}[Proof of Theorem \ref{theo_eigenvalue_C1a_C2a}]

Let $\lambda_j (\NP[0]^*)$ be eigenvalues of the zeroth mode operator $\NP[0]^*$ enumerated in the descending order. Then they are eigenvalues whose eigenfunctions are axially symmetric, namely, $\lambda_j (\NP[0]^*)=\rho_j (\NP^*)$.

We recall the decomposition \eqref{eq_sl_decomposition} in Proposition \ref{prop_sl_decomposition}. Taking a continuous global section $\xi: \partial \Sigma\to T^*\partial \Sigma$ such that $|\xi (p)|=1$ for all $p\in \partial \Sigma$, we have
\begin{align*}
    \frac{1}{2\pi}\int_{\partial\Sigma} \left(\left|-\frac{v_p}{4y|\xi(p)|}\right|+\left|-\frac{v_p}{4y|-\xi(p)|}\right|\right)\, \df \sigma (p)
    =\frac{1}{2\pi}\int_{\partial\Sigma} \frac{|v_p|}{2y}\,\df \sigma (p)
    =C_0.
\end{align*}
Thus, by Theorem \ref{theo_grubb}, we obtain 
\begin{equation}
    \label{eq_sl_pr_asympt}
    s_j \left(\Op \left(-\frac{v_p}{4y|\xi|}\right)\right) \sim C_0 j^{-1} \quad (j\to \infty). 
\end{equation}

Now we turn into the proof of the assertions \ref{enum_eigenvalue_C1a} and \ref{enum_eigenvalue_C2a}.

\ref{enum_eigenvalue_C1a} Suppose $\partial\Sigma$ is $C^{1, \alpha}$ for some $\alpha\in (0, 1)$. Then, by Proposition \ref{prop_sl_decomposition} and \eqref{eq_sl_pr_asympt}, we obtain $(v_p/2y)\SL[\partial\Sigma]\in \Scal^{1/\alpha}(\partial \Sigma)$. Now we apply Propositions \ref{theo_np_zero_mode_decomposition} and \ref{theo_nps_Sobolev} to the other terms in the decomposition \eqref{eq_operator_decomposition}. Then we obtain $\NP[0]^*\in \Scal^{1/\alpha+\varepsilon}(L^2 (\partial\Sigma))$ for all $\varepsilon>0$. We then infer from Lemma \ref{lemm_schatten_decay} \ref{enum_schatten_decay} that $|\lambda_j (\NP[0]^*, L^2 (\partial\Sigma))|=o(j^{-\alpha+\varepsilon})$ for all $\varepsilon>0$. This proves \eqref{1-100} since $\lambda_j (\NP[0]^*, L^2 (\partial\Sigma))=\lambda_j (\NP[0]^*, H^{-1/2}(\partial\Sigma))$ as we remarked in Subsection \ref{subs_schatten}.

\ref{enum_eigenvalue_C2a} Suppose $\partial\Sigma$ is $C^{2, \alpha}$ for some $\alpha\in (0, 1)$. We first prove
\begin{equation}
    \label{eq_singular_asymptotic_T}
    s_j(\Kcal_0^*)\sim C_0j^{-1}
\end{equation}
as $j\to \infty$. Here $C_0$ is the positive constant defined by \eqref{eq_C0}. By Propositions \ref{theo_np_zero_mode_decomposition}, \ref{theo_nps_Sobolev} and \ref{prop_sl_decomposition}, we have
\[
    \NP[0]^*-\Op \left(-\frac{v_p}{4y|\xi|}\right)=\NP[\partial\Sigma]^*+\Lcal+\Rcal^*\in \Scal^1(L^2(\partial\Sigma)).
\]
Hence we have $s_j(\NP[0]^*-\Op (-(v_p/4y|\xi|))) =o(j^{-1})$. Now we apply Theorem \ref{theo_ky_fan} and \eqref{eq_sl_pr_asympt} to obtain \eqref{eq_singular_asymptotic_T}.

We now apply Lemma \ref{lemm_schatten_decay} \ref{enum_singular_ev_decay} to obtain $|\lambda_j (\NP[0]^*)|=O(j^{-1})$ as $j\to \infty$. 
\end{proof}

\section{Proof of Theorem \ref{theo_np_ev_rot}}

\subsection{Principal symbol of the NP operator}

Throughout this section, we assume that $\partial\Omega$ is $C^\infty$. We do so since we employ the symbol calculus of $\Psi$DOs for the proof of Theorem \ref{theo_np_ev_rot}.

We abbreviate $\delta (p, p^\prime)$ to $\delta$. According to Proposition \ref{theo_ell_int_ae}, the integral kernels $K_0^*(p, p^\prime)$ and $S_0(p, p^\prime)$ of the zeroth mode NP operator and the zeroth mode single layer potential are expanded in $\delta=\delta (p, p^\prime)<1$ as
    \begin{align*}
        K_0^*(p, p^\prime)&=K_{\partial\Sigma}^* (p, p^\prime)\left(1+\frac{2(\Ell (\iu \delta)-(1+\delta^2)\Kell (\iu \delta))}{\pi (1+\delta^2)}\log \delta+\frac{\delta^2(g(\delta)-1)}{1+\delta^2}\right) \\
        &\quad -\frac{v_p}{4\pi y} \biggl(-\frac{2\Ell (\iu \delta)}{\pi (1+\delta^2)}\log \delta
        +\frac{4\log 2}{\pi}-1+\delta^2 \left(f(\delta)-\frac{g(\delta)-1}{1+\delta^2}\right)\biggr)
    \end{align*}
    and
    \[
        S_0 (p, p^\prime)=\frac{\Kell (\iu \delta)}{\pi^2}\log\delta-\frac{2\log 2}{\pi^2}-\frac{\delta^2 f(\delta)}{2\pi}
    \]
    where $f(z)$ and $g(z)$ are holomorphic functions. We observe that, if we define $\Fcal\subset L^1 ((0, 1))$ as
    \[
        \Fcal:=\{ h_1 (\delta)\log \delta+h_2 (\delta) \mid h_1(z), \ h_2(z) \text{ are holomorphic functions in } |z|<1\},
    \]
    then $K_0^*(p, p^\prime)$ and $S_0 (p, p^\prime)$ are of the form
    \[
        K_0^* (p, p^\prime)=K_{\partial\Sigma}^* (p, p^\prime)f_1(\delta (p, p^\prime))
        -\frac{v_p}{4\pi y}f_2 (\delta (p, p^\prime))
    \]
    and
    \[
        S_0 (p, p^\prime)=f_3 (\delta (p, p^\prime))
    \]
    where $f_j\in \Fcal$ ($j=1, 2, 3$). This implies that the operators $\NP[0]^*$ and $\SL[0]$ can be represented as $\Psi$DOs.

\begin{prop}
    \label{theo_symbols}
    If $\partial\Omega$ is $C^\infty$, then the principal symbol $\sigma^\mathrm{pr}(\NP[0])(x, \xi)$ of the operator $\NP[0]$ is
    \[
        \sigma^\mathrm{pr} (\NP[0]) (p, \xi)=-\frac{v_p}{4y |\xi|}.
    \]
\end{prop}

\begin{proof}
    If $\partial\Omega$ is $C^\infty$, then the integral kernel of $\NP[\partial\Sigma]^*$ belongs to $C^\infty (\partial\Sigma \times \partial\Sigma)$ (see \cite{FKMdr}). Thus, by Proposition \ref{theo_ell_int_ae} and \eqref{eq_operator_decomposition}, the principal symbol $\sigma^\mathrm{pr}(\NP[0])$ of $\NP[0]^*$ is calculated as
    \[
        \sigma^\mathrm{pr} (\NP[0]^*)(p, \xi)=\sigma^\mathrm{pr} \left(\frac{v_p}{2y}\SL[\partial\Sigma]\right)(p, \xi)=-\frac{v_p}{4y|\xi|},
    \]
    as desired.
\end{proof}

\subsection{Proof of Theorem \ref{theo_np_ev_rot}}

Since $-\SL[0]$ is a positive definite self-adjoint operator on $L^2(\partial\Sigma)$, it admits the complex power $(-\SL[0])^z$ for $z\in \Cbb$. Moreover, since $\SL[0]$ is a $\Psi$DO, the complex power $(-\SL[0])^z$ is also a $\Psi$DO with principal symbol $\sigma^\mathrm{pr} ((-\SL[0])^z)=\sigma^\mathrm{pr} (-\SL[0])^z=(2|\xi|)^{-z}$ by Proposition \ref{theo_symbols}. In particular, $(-\SL[0])^{1/2}$ is a unitary operator from $H^{-1/2}(\partial\Sigma)$ to $L^2(\partial\Sigma)$ and its inverse is $(-\SL[0])^{-1/2}$. Let
\begin{equation}
    \label{eq_np_symmetrized}
    \widehat{\Kcal}_0:=(-\Scal_0)^{1/2} \Kcal_0^* (-\Scal_0)^{-1/2}.
\end{equation}

\begin{prop}
    \label{theo_symmetrization_zero}
    Assume that $\partial\Sigma$ is $C^\infty$. Then the operator $\widehat{\Kcal}_0$ is self-adjoint on $L^2(\partial\Sigma)$ and we have
    \[
        \sigma(\widehat{\Kcal}_0, L^2(\partial\Sigma))=\sigma (\Kcal_0^*, H^{-1/2}(\partial\Sigma)).
    \]
\end{prop}

\begin{proof}
    One can easily see from \eqref{Plemelj_k} that
    \[
        (-\Scal_0)^{1/2} \Kcal_0^* (-\Scal_0)^{-1/2}=(-\Scal_0)^{-1/2}\Kcal_0 (-\Scal_0)^{1/2}=((-\Scal_0)^{1/2} \Kcal_0^* (-\Scal_0)^{-1/2})^*.
    \]
    This means that the $\widehat{\Kcal}_0$ is self-adjoint. Since $\NP[0]^*$ and $\widehat{\NP[]}_0$ are unitarily equivalent, their spectra coincide.
\end{proof}

We also recall the following Weyl asymptotics by Birman and Solomjak (recall the terminologies in Subsection \ref{subs_psiDO}): 

\begin{theo}[\cite{Birman-Solomjak77}, see also {\cite[Section 6]{Ponge23}}]\label{theo_BS}
    Let $M$ be a $C^\infty$-smooth one-dimensional compact manifold and $\Acal=\Op (a)$ be a compact self-adjoint $\Psi$DO with the principal symbol $a^\mathrm{pr}(x, \xi)$ which is locally in $C^\tau S^{-1}(\Rbb \times (\Rbb\setminus \{0\}))$ for some $\tau>0$. We denote the positive and negative eigenvalues of $\Acal$ as $\lambda^+_j (\Acal)$ and $-\lambda^-_j (\Acal)$ respectively and enumerate them in the descending order 
    \[
        \lambda^\pm_1(\Acal)\geq \lambda^\pm_2 (\Acal)\geq \cdots
    \]
    as long as they exist. Then 
    \[
        \lambda^\pm_j (\Acal)\sim\frac{1}{2\pi j}\int_M (a^\mathrm{pr}(x, \xi (x))_\pm +a^\mathrm{pr}(x, -\xi (x))_\pm)\, \df \sigma (x), 
    \]
    where $\xi: M\to T^*M$ is a continuous global section such that $|\xi(x)|=1$ for all $x\in M$ with respect to the metric $g$ and 
    \[
        t_\pm:=
        \begin{cases}
            |t| & \text{if } \pm t\geq 0, \\
            0 & \text{if } \pm t\leq 0
        \end{cases}
    \]
    for $t\in \Rbb$. 
\end{theo}
Now we are ready to prove Theorem \ref{theo_np_ev_rot}.

\begin{proof}[Proof of Theorem \ref{theo_np_ev_rot}]
    We denote the $j$-th positive and negative eigenvalue of $\Kcal_0^*$ by $\lambda^+_j(\Kcal_0^*)$ and $-\lambda^-_j (\Kcal_0^*)$, respectively. Since $\Kcal_0^*$ is the restriction of the NP operator $\NP^*$ to axially symmetric functions, we have
    \begin{equation}
        \label{eq_NP_NP0_pm_ev}
        \rho^\pm_j (\NP^*)=\lambda^\pm_j(\Kcal_0^*)=\lambda^\pm_j (\widehat{\Kcal}_0)
    \end{equation}
    by Proposition \ref{theo_symmetrization_zero}.

    Similarly, we denote the $j$-th eigenvalue of $\Kcal_0^*$ by $\lambda_j(\Kcal_0^*)$ in the sense of ordering
    \[
        |\lambda_1(\Kcal_0^*)|\geq |\lambda_2(\Kcal_0^*)|\geq \cdots.
    \]
    Then we have
    \begin{equation}
        \label{eq_NP_NP0_abs_ev}
        |\rho_j (\NP^*)|=|\lambda_j (\Kcal_0^*)|=|\lambda_j (\widehat{\Kcal}_0)|.
    \end{equation}

    By the definition of $\widehat{\NP[]}_0$, Proposition \ref{theo_symmetrization_zero} and the symbol calculus, $\widehat{\NP[]}_0$ is a self-adjoint $\Psi$DO of order $-1$ which has the same principal symbol $a^\mathrm{pr}(p, \xi):=-v_p/(4y|\xi|)$ as the operator $\NP[0]^*$. Thus we can employ Theorem \ref{theo_BS} and obtain
    \begin{equation}\label{eq_BS_conclusion}
        \lambda^\pm_j(\widehat{\NP[]}_0)\sim C^\pm_0 j^{-1}, \quad
        |\lambda_j(\widehat{\NP[]}_0)|\sim C_0 j^{-1}
    \end{equation}
    where
    \begin{align*}
        C_0^\pm :=&\,\frac{1}{2\pi}\int_{\partial\Sigma} (a^\mathrm{pr}(p, \xi (p))_\pm+a^\mathrm{pr}(p, -\xi (p))_\pm)\, \df \sigma (p) \\
        =&\,\mp\frac{1}{4\pi}\int_{\mp v_p\geq 0} \frac{v_p}{y}\, \df \sigma (p)
    \end{align*}
    and
    \[
        C_0:=\frac{1}{4\pi}\int_{\partial\Sigma} \frac{|v_p|}{y}\, \df \sigma (p).
    \]
    Thus we have
    \eqref{eq_Cpm0} and \eqref{eq_C0}.
\end{proof}

\section{Proof of Theorem \ref{theo_essential_spectrum}}

We compare the essential spectra of $\NP[\partial\Sigma]^*$ and $\NP[0]^*$. Since $\NP[\partial\Sigma]^*$ and $\NP[0]^*$ are realized as self-adjoint operators with respect to the inner products $\jbk{\cdot, \cdot}_{*}=\jbk{\cdot, \cdot}_{*, \partial\Sigma}$ and $\jbk{\cdot, \cdot}_0$ respectively, we introduce the notation
\begin{align*}
    \Hcal_{\partial\Sigma}^*:=(H^{-1/2}(\partial\Sigma), \jbk{\cdot, \cdot}_{*}),  \quad
    \Hcal_0^*:=(H^{-1/2}(\partial\Sigma), \jbk{\cdot, \cdot}_0)
\end{align*}
in order to avoid possible confusion. Let $\| \cdot \|_{*}$ and $\| \cdot \|_{0}$ respectively denote norms on $\Hcal_{\partial\Sigma}^*$ and $\Hcal_0^*$. We emphasize that both are equivalent to the $H^{-1/2}(\partial\Sigma)$-norm.

We obtain the following theorem.
\begin{theo}
    \label{theo_np_np0_ess_spec}
    If $\partial \Sigma$ is Lipschitz, then it holds that
    \begin{equation}
        \label{eq_np_np0_ess_spec}
        \sigma_\mathrm{ess}(\NP[\partial\Sigma]^*, \Hcal_{\partial\Sigma}^*)
        =\sigma_\mathrm{ess}(\NP[0]^*, \Hcal_0^*).
    \end{equation}
\end{theo}

Actually, this is an immediate consequence of the following variant of the Weyl theorem:

\begin{prop}\label{lemm_weyl_variant}
    Let $\Hcal$ be a Banach space and $\jbk{\cdot, \cdot}_1$ and $\jbk{\cdot, \cdot}_2$ be inner products on $\Hcal$ which induce norms equivalent to the original norm on $\Hcal$ and thus both of $\Hcal_1=(\Hcal, \jbk{\cdot, \cdot}_1)$ and $\Hcal_2=(\Hcal, \jbk{\cdot, \cdot}_2)$ are Hilbert spaces. If $\Acal_1$ and $\Acal_2$ are bounded linear operators on $\Hcal$ such that
    \begin{itemize}
        \item $\Acal_1$ and $\Acal_2$ are self-adjoint operators on $\Hcal_1$ and $\Hcal_2$, respectively;
        \item $\Acal_1-\Acal_2$ is compact on $\Hcal$,
    \end{itemize}
    then
    \[
        \sigma_\mathrm{ess} (\Acal_1, \Hcal_1)
        =\sigma_\mathrm{ess} (\Acal_2, \Hcal_2).
    \]
\end{prop}

\begin{proof}
    Let $\|\varphi\|_j:=\jbk{\varphi, \varphi}_j^{1/2}$ be the norm induced by the inner product $\jbk{\cdot, \cdot}_j$.

    Let $\lambda \in \sigma_\mathrm{ess}(\Acal_1, \Hcal_1)$ and let $\{\varphi_n\}_{n=1}^\infty$ be the Weyl sequence of $\lambda$ and $\Acal_1$, that is, $\| \varphi_n \|_1=1$ for all $n$, $\varphi_n$ weakly converges to $0$, and $\| \Acal_1 [\varphi_n]-\lambda \varphi_n\|_1\to 0$. Let $\psi_n:=\varphi_n/\|\varphi_n\|_2$. We show that a subsequence of $\psi_n$ makes a Weyl sequence of $\lambda$ and $\Acal_2$, and hence $\lambda \in \sigma_\mathrm{ess} (\Acal_2, \Hcal_2)$.

    It is obvious by the definition that $\|\psi_n\|_2=1$ for all $n$. Let $\varphi\in \Hcal$. By the Riesz representation theorem, there exists the unique $\Bcal[\varphi]\in \Hcal$ such that $\jbk{\psi, \varphi}_2=\jbk{\psi, \Bcal [\varphi]}_1$ for all $\psi \in \Hcal$, and
    \[
        \|\Bcal [\varphi]\|_1=\sup_{\psi\in \Hcal\setminus \{ 0 \}}\frac{|\jbk{\psi, \varphi}_2|}{\|\psi\|_1}\approx \|\varphi\|_2 \approx \|\varphi\|_1.
    \]
    In particular, $\Bcal$ is a bounded invertible operator on $\Hcal$. Thus, for any $\varphi\in \Hcal$, we have
    \[
        |\jbk{\psi_n, \varphi}_2|=\frac{|\jbk{\varphi_n, \Bcal[\varphi]}_1|}{\|\varphi_n\|_2} \approx |\jbk{\varphi_n, \Bcal[\varphi]}_1|\to 0
    \]
    as $n\to \infty$. 
    
    Since $\Acal_1-\Acal_2$ is compact and $\varphi_n$ weakly converges to $0$, there is a subsequence of $\varphi_n$, which we still denote by $\varphi_n$, such that $(\Acal_1-\Acal_2) [\varphi_n]$ converges to $0$. 
    We then see $\| \Acal_2 [\psi_n]-\lambda \psi_n\|_2 \to 0$ as $n \to \infty$ from the inequality
    \begin{align*}
        \| \Acal_2 [\psi_n]-\lambda \psi_n\|_2
        &\leq \| (\Acal_1-\Acal_2) [\psi_n]\|_2+\| \Acal_1 [\psi_n]-\lambda \psi_n\|_2 \\
        &\approx \| (\Acal_1-\Acal_2) [\varphi_n]\|_1+\| \Acal_1 [\varphi_n]-\lambda \varphi_n\|_1 .
    \end{align*}
    Thus $\psi_n$ is a Weyl sequence of $\lambda$ and $\Acal_2$.

    So far, we proved that $\sigma_\mathrm{ess}(\Acal_1, \Hcal_1)\subset \sigma_\mathrm{ess}(\Acal_2, \Hcal_2)$. By changing the roles of $\Acal_1$ and $\Acal_2$, we can prove the opposite inclusion relation.
\end{proof}

Now we prove Theorem \ref{theo_np_np0_ess_spec}.

\begin{proof}[Proof of Theorem \ref{theo_np_np0_ess_spec}]
    We already observed that the operator $\NP[0]^*$ and $\NP[\partial\Sigma]^*$ are self-adjoint on $\Hcal_0^*$ and $\Hcal_{\partial\Sigma}^*$, respectively, and that the norms on Hilbert spaces $\Hcal_0^*$ and $\Hcal_{\partial\Sigma}^*$ are equivalent. We claim that the operator $\NP[0]^*-\NP[\partial\Sigma]^*$ is compact on $H^{-1/2}(\partial\Sigma)$ in order to apply Lemma \ref{lemm_weyl_variant}. In fact, since $\Rcal^*$ is compact on $H^{-1/2}(\partial\Sigma)$ (Corollary \ref{coro_rs_lipschitz}) and the composition
    \begin{align*}
        f\in H^{-1/2}(\partial\Sigma)
        &\longrightarrow \SL[\partial\Sigma][f]\in H^{1/2}(\partial\Sigma) \longrightarrow \frac{v_p}{4\pi y}\SL[\partial\Sigma][f]\in L^2 (\partial\Sigma) \\
        &\longhookrightarrow \frac{v_p}{4\pi y}\SL[\partial\Sigma][f]\in H^{-1/2}(\partial\Sigma)
    \end{align*}
    is compact by the Rellich-Kondrachov embedding theorem, the decomposition \eqref{eq_a0_asymptotic_c1a} shows that $\NP[0]^*-\NP[\partial\Sigma]^*$ is compact on $H^{-1/2}(\partial\Sigma)$.

    Now we apply Lemma \ref{lemm_weyl_variant} and obtain the conclusion.
\end{proof}

Now we prove Theorem \ref{theo_essential_spectrum}.

\begin{proof}[Proof of Theorem \ref{theo_essential_spectrum}]
    According to Theorem \ref{theo_np_np0_ess_spec}, it suffices to prove the inclusion
    \begin{equation}
        \label{eq_nps_np3_ess_spec}
        \sigma_\mathrm{ess}(\NP[0]^*, \Hcal_0^*)
        \subset \sigma_\mathrm{ess}(\NP^*, H^{-1/2}(\partial\Omega)).
    \end{equation}

    Let $\Ucal: H^{-1/2}(\partial\Omega)\to H^{-1/2}(\partial\Sigma \times S^1)$ be the operator defined in \eqref{eq_defi_j}. Assume $\lambda\in \sigma_\mathrm{ess}(\NP[0]^*, \Hcal_0^*)$ and let $\{f_n\}_{n=1}^\infty$ be a Weyl sequence of $\lambda$ and $\NP[0]^*$. We show that $g_n:=\Ucal^{-1}[f_n]\in H^{-1/2}(\partial\Omega)$ yields a Weyl sequence of $\lambda$ and $\NP^*$. In fact, by Corollary \ref{coro_inner_product_decomposition} \ref{enum_inner_product_mode}, we have
    \[
        \| \Ucal^{-1}[f_n]\|_*=\|f_n\|_0=1.
    \]

    Let $h\in H^{-1/2}(\partial\Omega)$. It follows from \eqref{eq_J_unitary} and Proposition \ref{prop_single_layer_separation} that
    \begin{align*}
        \jbk{g_n, h}_*&= -\jbk{\SL \Ucal^{-1}[f_n], h}_{\partial\Omega} \\
        &= -\jbk{\Ucal^{-1}\SL[0][f_n], h}_{\partial\Omega}
        =-\jbk{\SL[0][f_n], \Ucal [h]}_{\partial\Sigma\times S^1} .
    \end{align*}
    Note that $\int_{S^1} \Ucal [h](\cdot, \varphi)\, \df \varphi \in H^{-1/2}(\partial\Sigma)$ and hence
    \begin{align*}
        -\jbk{\SL[0][f_n], \Ucal [h]}_{\partial\Sigma\times S^1} = \jbk{ f_n, \int_{S^1} \Ucal [h](\cdot, \varphi)\, \df \varphi }_0 \to 0
    \end{align*}
    as $n\to \infty$. Thus $\{g_n\}_{n=1}^\infty$ weakly converges to $0$ in $H^{-1/2}(\partial\Omega)$.

    By \eqref{eq_nps_separated} and Corollary \ref{coro_inner_product_decomposition} \ref{enum_inner_product_mode}, we have
    \begin{align*}
        \| (\NP^*-\lambda I)[g_n]\|_*^2
        = \| \Ucal^{-1}(\NP[0]^*-\lambda I)[f_n]\|_*^2
        =\| (\NP[0]^*-\lambda I)[f_n]\|_0^2 \to 0
    \end{align*}
    as $n\to \infty$. Thus, $\{g_n\}_{n=1}^\infty$ is a Weyl sequence of $\lambda$ and $\NP^*$. This implies $\lambda\in \sigma_\mathrm{ess}(\NP^*, H^{-1/2}(\partial\Omega))$.
\end{proof}

\appendix

\section{Proof of Proposition \ref{theo_ell_int_ae}}\label{sec_proof_elliptic_behavior}

As preliminaries, we introduce the digamma function 
\[
    \psi (z):=\frac{\Gamma^\prime (z)}{\Gamma (z)}
\]
where $\Gamma (z)$ is the Gamma function 
\[
    \Gamma (z):=\int_0^\infty t^{z-1}\e^{-t}\, \df t
\]
for $\Re z>0$. 

We introduce the power series 
\begin{equation}\label{eq_f}
    f(z):=\frac{2}{\pi}\sum_{n=1}^\infty \left(\frac{(2n-1)!!}{(2n)!!}\right)^2 \left(\psi (n+1)-\psi \left(n+\frac{1}{2}\right)\right) (-1)^nz^{2n-2}
\end{equation}
for $z\in \Cbb$ with $|z|<1$, which converges by virtue of the inequality 
\[
    \limsup_{n\to \infty} |\psi (n+a)|^{1/n}\leq 1
\]
for any $a\in \Cbb$, which is proved by the identity $\psi (z+1)=\psi (z)+1/z$ for any $z\in \Cbb\setminus \{0\}$ 
\cite[\href{https://dlmf.nist.gov/5.5.E2}{(5.5.2)}]{NIST}.

\begin{proof}[Proof of Proposition \ref{theo_ell_int_ae}]
    The identity \eqref{eq_full_expansion_first} with $f(z)$ defined by \eqref{eq_f} is obtained by combining the Gauss hypergeometric series representation of the elliptic integral \cite[\href{https://dlmf.nist.gov/19.5.E1}{(19.5.1)}]{NIST}, a formula of the Gauss hypergeometric series with the reciprocal argument \cite[\href{https://dlmf.nist.gov/15.8.E8}{(15.8.8)}]{NIST}, and special values and the reflection formula of the digamma function \cite[\href{https://dlmf.nist.gov/5.4.E12}{(5.4.12)}, \href{https://dlmf.nist.gov/5.4.E13}{(5.4.13)} and \href{https://dlmf.nist.gov/5.5.E4}{(5.5.4)}]{NIST}. 

    We define 
    \[
        g(z):=1+(1+z^2)\left(\frac{2\Kell (\iu z)/\pi-1}{z^2}-zf^\prime (z)-2f(z)\right)
    \]
    for $z\in \Cbb$ with $|z|<1$. The Taylor expansion of $\Kell (z)$ at $z=0$ implies that $z=0$ is a removable singularity of $g(z)$. Then, by the formula \cite[\href{https://dlmf.nist.gov/19.4.E1}{(19.4.1)}]{NIST}  
    \begin{equation}
        \label{eq_second_derivative}
        \Ell (z)=(1-z^2)\Kell (z)+z(1-z^2)\frac{\df}{\df z}\Kell (z)
    \end{equation}
    for $z\in \Cbb \setminus ((-\infty, -1]\cup [1, \infty))$ and \eqref{eq_full_expansion_first}, we obtain 
    \begin{align*}
        \Ell (\iu \delta^{-1})
        &=\frac{1+\delta^2}{\delta^2}\Kell (\iu \delta^{-1})+\frac{\iu (1+\delta^2)}{\delta^3}\left.\frac{\df \Kell}{\df z}\right|_{z=\iu \delta^{-1}} \\
        &=\frac{1+\delta^2}{\delta^2}\Kell (\iu \delta^{-1})-\frac{1+\delta^2}{\delta}\frac{\df}{\df \delta}[\Kell (\iu \delta^{-1})] \\
        &=\frac{2(1+\delta^2)\log \delta}{\pi}\frac{\df}{\df \delta}[\Kell (\iu \delta)]+\frac{1}{\delta}+\delta g(\delta). 
    \end{align*}
    We employ \eqref{eq_second_derivative} for $\df \Kell (\iu \delta)/\df \delta$ to obtain \eqref{eq_full_expansion_second}. 
\end{proof}

\section*{Acknowledgment}\addcontentsline{toc}{section}{\numberline{}Acknowledgment}
The authors express deep gratitude to Professor Yoshihisa Miyanishi for fruitful discussions, and to the referee for valuable comments and suggestions. 







\end{document}